\documentclass[english,10pt]{amsart}
\usepackage[english]{babel}

\usepackage[matrix,arrow]{xy}
\xyoption{all}
\usepackage{amscd,amssymb,amsfonts,amsmath}

\usepackage{graphics}
\usepackage{epsfig}
\usepackage{mathrsfs}

\newcommand\mypagesizel{
\textwidth= 6.5in
\textheight=9in
\voffset-.55in
\hoffset -0.75in
\marginparwidth=56pt
}

\newcommand{\p}[0]{{\mathbb P}}

\newcommand{\N}{\textup{N}}
\newcommand{\Z}{\textup{Z}}

\newcommand{\NE}{\overline{\textup{NE}}}

\renewcommand{\phi}{\varphi}

\newcommand{\B}{\textup{B}}

\newcommand{\sL}{\mathscr{L}}

\newcommand{\sO}{\mathscr{O}}

\mypagesizel

\newtheorem{thm}{Theorem}[section]

\newtheorem{lemma}[thm]{Lemma}

\newtheorem{prop}[thm]{Proposition}

\newtheorem*{thm*}{Theorem}

\theoremstyle{definition}

\newtheorem{const}[thm]{Construction}

\newtheorem{defn-thm}[thm]{Definition-Theorem} 
\newtheorem{defn-lemma}[thm]{Definition-Lemma}

\theoremstyle{remark}

\newtheorem*{not-and-def}{Notation and definitions}

\numberwithin{equation}{section}

\begin{document}

\title[]{Fano varieties with $\mathrm{Nef}(X)=\mathrm{Psef}(X)$ and $\rho(X)=\mathrm{dim}\, X-1$}

\author{Wenhao OU} 

\email{wenhaoou@math.ucla.edu}
\address{UCLA Mathematics Department,
520 Portola Plaza, Los Angeles, CA 90095, USA}


\begin{abstract}
We classify mildly singular Fano varieties $X$ such that $\mathrm{Nef}(X)=\mathrm{Psef}(X)$ and $\rho(X)=\mathrm{dim}\, X-1$.
\end{abstract}

\maketitle

\tableofcontents

\section{Introduction}

During the last few decades, a lot of progress has been made in   classification of complex projective varieties. In particular, from the viewpoint of the Minimal Model Program, there are three building blocks in birational classification theory. They are canonically polarized varieties, Calabi-Yau varieties and Fano varieties, which correspond respectively to varieties with ample canonical divisors, varieties with trivial canonical divisors, and varieties with anti-ample canonical divisors. 

In this paper, we are interested in Fano varieties. According to \cite[Thm. 0.2]{KMM92a}, smooth complex  Fano varieties of fixed dimension form a bounded family. In particular, there are finitely many smooth Fano varieties of fixed dimension up to deformation equivalence. The classification of  smooth Fano varieties of dimension $2$, or equivalently, of Del Pezzo surfaces, is well-known: a Del Pezzo surface is either a product of two $\p^1$, or is a blowup of $\p^2$ at at most $8$ points. In dimension three, the list is much longer, but we still have a complete classification up to deformation equivalence thanks to the work of Iskovskikh and of Mori-Mukai (see \cite{IPY99} for a summary).  Starting from dimension four, the classification of Fano varieties seems to be a very difficult problem. 

Nevertheless, if we impose some additional conditions, a classification  might be possible. One famous problem is a conjecture of Campana-Peternell, which states that every smooth complex Fano variety with nef tangent bundle should be a product of  rational homogeneous spaces. This problem is still largely open. It is known that if a smooth Fano  variety  $X$ has nef tangent bundle, then it must satisfy the condition  $\mathrm{Nef}(X) = \mathrm{Psef}(X)$ (see \cite[Prop. 3.1]{MOSWW15} and \cite[Lem. 4.4]{Dru16}), where   $\mathrm{Nef}(X)$ is  the cone of nef divisors in $X$, and   $\mathrm{Psef}(X)$ is the cone of pseudo-effective divisors.

If  $X$ is a Fano variety  which satisfies   $\mathrm{Nef}(X) = \mathrm{Psef}(X)$, and if moreover $X$ has $\mathbb{Q}$-factorial log canonical singularities, then the Picard number of $X$, is at most equal to the dimension of $X$  (see \cite[Lem. 4.9]{Dru16}). If we assume further that $\rho(X) = \mathrm{dim}\, X$  and that  $X$ has locally factorial canonical singularities, then $X$ is a product of double covers of $\p^1\times \cdots \times \p^1$. More precisely, Druel proves the following theorem (see {\cite[Thm. 1.1]{Dru16}} and {\cite[Prop. 10.4]{Dru16}}).

\begin{thm}
\label{thm-double-cover-p1}
Let $X$ be a Fano variety with locally factorial canonical singularities such that $\rho(X) = \mathrm{dim}\, X$ and  $\mathrm{Nef}(X) = \mathrm{Psef}(X)$. Then $X\cong X_1\times \cdots \times X_k$ such that  for all $i=1,...,k$, either $X_i\cong \p^1$,  or $\mathrm{dim}\, X_i\geqslant 3$ and $X_i$ is a double cover  of $\p^1\times \cdots \times \p^1$, branched along a prime divisor of degree $(2,...,2)$.
\end{thm}

In this paper, we will consider Fano  varieties $X$ with locally factorial canonical singularities such that $X$ is smooth in codimension $2$, $\rho(X) = \mathrm{dim}\, X-1$ and  $\mathrm{Nef}(X) = \mathrm{Psef}(X)$. If the dimension of $X$ is $2$, then $X$ is isomorphic to $\p^2$. In  dimension $3$, we prove the following   theorem.

\begin{thm}
\label{thm-classification-3f}
Let $X$ be a Fano threefold with isolated locally factorial canonical singularities such that $\rho(X) = 2$ and  $\mathrm{Nef}(X) = \mathrm{Psef}(X)$. Then one of the following holds.

\begin{enumerate}
\item[1.] $X\cong \p^1\times \p^2$.
\item[2.] $X$ is a double cover of $\p^1\times \p^2$, branched along a prime divisor of degree $(2,2)$.
\item[3.] $X$ is a double cover of $\p^1\times \p^2$, branched along a prime divisor of degree $(2,4)$.
\item[4.] $X$ is a   hypersurface in $\p^2\times \p^2$ of degree $(1,1)$.
\item[5.] $X$ is a hypersurface in $\p^2\times \p^2$ of degree $(1,2)$.
\item[6.] $X$ is a hypersurface in $\p^2\times \p^2$ of degree $(2,2)$.
\item[7.] $X$ is a double cover of a smooth hypersurface $Y$ in $\p^2\times \p^2$ of degree $(1,1)$, branched along a prime divisor $D$, which is the intersection of $Y$ and a hypersurface of    degree $(2,2)$ in $\p^2\times \p^2$.
\end{enumerate}
\end{thm}

We note that if we assume that $X$ is smooth, then Theorem \ref{thm-classification-3f} follows from the classification of smooth Fano threefolds due to Mori and Mukai (\cite{MM81}).  In all dimensions, we prove the following classification theorem.

\begin{thm}
\label{thm-classification}
Let $X$ be a  Fano variety of dimension at least $2$,   with locally factorial canonical singularities such that $X$ is smooth in codimension $2$, $\rho(X) = \mathrm{dim}\, X-1$ and  $\mathrm{Nef}(X) = \mathrm{Psef}(X)$.  Then $X\cong X_1\times X_2$ such that
\begin{enumerate}
\item[-] $X_1$ is either a point or a variety satisfying the condition  of Theorem \ref{thm-double-cover-p1};
\item[-] $X_2$ is a finite cover, of degree $1, 2 $ or $4$, over $(\p^1)^s  \times \p^2$ or  $(\p^1)^s  \times W$, where $s\geqslant 0$ is an integer, and $W$ is a normal hypersurface in $\p^2\times \p^2$. 
\end{enumerate}
More precisely, $X_2$ is either $\p^2$, or one of the varieties in Theorem \ref{thm-classification-3f}, or a variety obtained by  Construction \ref{exmp-112},   or a variety  obtained by Construction \ref{exmp-122}.
\end{thm}

This   theorem  is false without assuming the variety  is smooth in codimension $2$. In fact, even in dimension $2$, there are surfaces $X$ other than $\p^2$ which have Picard number $1$ and locally factorial canonical singularities. They are Fano surfaces which have exactly one singular point which is canonical of type $E_8$ (see \cite[Lem. 6]{MZ88}).

\textit{Outline of the proof of the theorems.} Let $X$ be a   variety satisfying the condition in Theorem \ref{thm-classification}.   We first consider the case when $\mathrm{dim}\, X=3$. In this case, the  Mori cone $\NE(X)$ has exactly two extremal rays $R_1,R_2$. Let $f_i \colon X\to B_i$ be the extremal contraction with respect to $R_i$ for $i=1,2$. We can prove that $B_i$ is either $\p^1$ or $\p^2$ and the product $f_1\times f_2$ is finite onto its image. Hence $X$ is a finite cover of $\p^1\times \p^2$ or a finite cover of a hypersurface of $\p^2\times \p^2$. With the help of two results on finite morphisms between Fano threefolds, we can deduce the proof of  Theorem \ref{thm-classification-3f}.

After this, we consider the case when $n=\mathrm{dim}\, X\geqslant 4$. The   proof of Theorem \ref{thm-classification} is by induction on the dimension of $X$.  We will first  prove that  there is a fibration $g \colon X\to \p^1$ since $n\geqslant 4$. For this property, we can reduced to the case when $n=4$. The proof of the four dimensional case is given in section \ref{Fano fourfolds satisfying $(*)$}.

There is a Mori fibration $h \colon X\to Y$ such that $h\times g \colon X \to Y\times \p^1$ is a finite surjective morphism. We can prove further that $Y$  also  satisfies the condition in Theorem \ref{thm-classification}.  By induction, we assume that  Theorem \ref{thm-classification} is true in dimension smaller than $n$. In particular, $h\times g$ induces a finite surjective morphism $f \colon X\to Z$ such that   $Z$ is either $\p^1\times \cdots \times \p^1\times \p^2$ or $\p^1\times \cdots \times \p^1\times W$, where $W$ is a normal ample hypersurface in $\p^2\times \p^2$.  In the end, we will   conclude the proof of  Theorem \ref{thm-classification} by studying some finite morphisms between Fano varieties.

\textbf{Acknowledgement}: The author would like to thank his advisor St\'ephane Druel for suggesting  this problem and for his help. He is also very grateful to the referee for reading  this paper carefully  and  providing  enormous helpful remarks.

\section{Notation and preliminaries}

In this paper, we will work over $\mathbb{C}$, the field of complex numbers. A variety is a separated integral scheme of finite type defined over $\mathbb{C}$.   If $X$ is a variety, then we denote by $(X)^k$ the product of $k$ copies of $X$.  

A normal variety $X$ is called locally factorial (resp. $\mathbb{Q}$-factorial) if for any Weil divisor $D$ in $X$, $D$ (resp. some multiple of $D$)  is a Cartier divisor.  We refer to \cite[\S 2.3]{KM98} for the definition of klt singularities and  canonical singularities. We note that locally factorial  klt  singularities are canonical. We denote by $K_X$ a canonical divisor of $X$. The variety $X$ is said to be Gorenstein if $X$ is Cohen-Macaulay and $K_X$ is a Cartier divisor. If $D$ and $E$ are two $\mathbb{R}$-divisors in $X$ which are numerically equivalent, then we write $D\equiv  E$. Let $\mathrm{Pic}(X)$ be the Picard group of $X$ and let $\mathrm{N}^1(X)$ be the vector space $\mathrm{Pic}(X)\otimes \mathbb{R}/{\equiv}$.  We denote the real vector space of $1$-cycles modulo numerical equivalence by $\mathrm{N}_1(X)$. These two spaces, $\N_1$ and $\N^1$,  are dual to each other by the intersection pairing.  The dimension $\rho(X)$ of $\mathrm{N}^1(X)$ is called the Picard number of $X$.  Let $\mathrm{Nef}(X)$ be the cone of nef divisors in $X$. Let $\mathrm{Psef}(X)$ be the cone of pseudo-effective divisors, that is, the closed convex cone generated by effective divisors. Then we have $$\mathrm{Nef}(X) \subseteq \mathrm{Psef}(X) \subseteq \N^1(X).$$ Let $\NE(X)$ be the  Mori cone, that is, the closed convex cone generated by effective $1$-cycles in $\N_1(X)$. Then $\NE(X)$ is dual to the cone $\mathrm{Nef}(X)$.

A normal projective variety $X$ is called a Fano variety if its anti-canonical divisor $-K_X$ is a $\mathbb{Q}$-ample divisor, that is, some multiple  of $-K_X$ is a Cartier ample divisor. If $X$ is a Fano variety with $\mathbb{Q}$-factorial klt singularities, then the cone $\NE(X)$ is polyhedral by the Cone Theorem (see \cite[\S 3]{KM98}).

A fibration $f \colon X\to Y$ between normal varieties is a proper surjective morphism with connected fibers. The relative dimension of $f$ is   $\mathrm{dim}\, X - \mathrm{dim}\, Y$. The discriminant of $f$ is the largest closed subset of $Y$ over which $f$ is not smooth.  A projective fibration $f \colon X\to Y$ is said to be a Fano fibration if some multiple of $-K_X$ is  Cartier and  relatively ample. 

Let  $f \colon X\to Y$ be a Fano fibration between projective normal varieties such that $X$ has $\mathbb{Q}$-factorial klt singularities.  The relative Picard number of $f$ is the dimension of the closed convex cone in $\N_1(X)$ generated by the curves contracted by $f$. This number is equal to $\rho(X)-\rho(Y)$ (see \cite[Lem. 3-2-5]{KMM87}). The Fano fibration $f$ is called a Mori fibration if its relative Picard number is $1$.

\section{Double covers between normal varieties}
We recall the notion of cyclic covers. Let $Y$ be a normal variety and $\sL$ a line bundle on $Y$.  Assume that there is a positive integer $k$ and a section $s\in H^0(Y,\sL^{\otimes k})$.  Let $D$ be the  subscheme defined by $s=0$. Assume that $D$ is reduced. Then there is a cyclic cover $g \colon Z\to Y$ with respect to the isomorphism $\sL^{\otimes k} \cong \sO_Y(D)$ induced by $s$.  The morphism $g$ is branched exactly along $D$ and $g^*D=k \cdot \mathrm{Supp}\, (g^*D)$ (see \cite[Def. 2.50]{KM98}). By construction, $Z$ may not be normal. However, we have the following result.

\begin{lemma}
\label{lem-double-cover-normal}
Let $Y$ be a normal Cohen-Macaulay variety and let $D$ be a reduced divisor. Assume that there is a line bundle $\sL$ on $Y$ such that $\sL^{\otimes k}\cong \sO_Y(D)$ for some $k>0$. If $f \colon Z\to Y$ is the corresponding cyclic cover, then $Z$ is  normal and Cohen-Macaulay.
\end{lemma}

\begin{proof}
We may assume that $Y=\mathrm{Spec}\, A$ is affine. Then $Z\cong \mathrm{Spec}\, A[T]/(T^k-s)$, where $T$ is an indeterminate and $s=0$ is the equation defining $D$. Since $Y$ is Cohen-Macaulay, so is $\mathrm{Spec}\, A[T]$. Thus $Z$ is also Cohen-Macaulay for it is a Cartier divisor in $\mathrm{Spec}\, A[T]$. Hence, we only need to prove that $Z$ is smooth in codimension $1$. However, there is an open subset $Y_0$ of $Y$ such that $\mathrm{codim}\, Y\backslash Y_0 \geqslant 2$ and both $Y_0$ and $D|_{Y_0}$ are smooth. Thus $Z_0=g^{-1}(Y_0)$ is smooth and $\mathrm{codim}\, Z\backslash Z_0 \geqslant 2$. This shows that $Z$ is normal.
\end{proof}

\begin{lemma}
\label{lem-fiber-product-normal}
Let $f \colon X\to Y$ be an equidimensional fibration between normal Cohen-Macaulay varieties. Assume that there is an open subset $U$ of $Y$ whose complement has codimension at least $2$ such that $f$ has reduced fiber over $U$. Let  $g \colon Y'\to Y$ be a cyclic cover with $Y'$  normal. Let $X'$ be the fiber product $X\times_Y Y'$. Then $X'$ is a normal Cohen-Macaulay  variety.
\end{lemma}

\begin{proof}
Assume that $g$ is the cyclic cover with respect to some isomorphism $\sL^{\otimes k} \cong \sO_Y(D)$, where $D$ is a  Cartier divisor and $\sL$ is a line bundle on $Y$. Since $Y'$ is normal, $D$ is reduced. Then the natural morphism $X'\to X$ is the cyclic cover with respect to the isomorphism $(f^*\sL)^{\otimes k} \cong \sO_X(f^*D)$. Since $f$ is equidimensional and has reduced fibers over $U$, $f^*D$ is a reduced divisor in $X$.  Since $X$ is Cohen-Macaulay, by Lemma \ref{lem-double-cover-normal}, $X'$ is normal Cohen-Macaulay.
\end{proof}

We will prove some properties of double covers.

\begin{lemma}
\label{lem-double-cover}
Let $f \colon X\to Y$ be a double cover between normal varieties with $Y$   locally factorial. Let $D$ be the codimension $1$ part of the discriminant of $f$. Then  there is a line bundle $\sL$ on $Y$ such that $\sL^{\otimes 2}\cong \sO_Y(D)$. If $g \colon Z\to Y$ is the corresponding cyclic cover, then $X$ is the normalization of $Z$.  Moreover, if $Y$ is Cohen-Macaulay, then  $X$ is smooth if and only if both $D$ and $Y$ are smooth.
\end{lemma}

\begin{proof}
Let $Y_0$ be the largest open subset contained in the smooth locus of $Y$ such that both its preimage $X_0=f^{-1}(Y_0)$ and the divisor $D|_{Y_0}$ are smooth. Then $\mathrm{codim}\, Y\backslash Y_0 \geqslant 2$ and there is a line bundle $\sL_0$ on $Y_0$ such that $f|_{X_0}$ is the cyclic cover with respect to  $\sL_0^{\otimes 2}\cong (\sO_Y(D))|_{Y_0}$ (See \cite[\S 0.1]{CD89}).  Since $Y$ is locally factorial, there is a line bundle $\sL$ on $Y$ such that $\sL|_{Y_0}\cong \sL_0$. Then $\sL^{\otimes 2}\cong \sO_Y(D)$. Let $g \colon Z\to Y$ be the corresponding cyclic cover. Since $X$ is normal, it is the normalization of $Z$.

If $Y$ is  Cohen-Macaulay, then $X\cong Z$ by Lemma \ref{lem-double-cover-normal}. Thus $X$ is smooth if and only if $D$ and $Y$ are smooth (see \cite[Lem. 2.51]{KM98}).
\end{proof}

\begin{lemma}
\label{lem-double-cover-Fano}
Let $f \colon X\to Y$ be a double cover between normal  varieties. Assume that $Y$ is $\mathbb{Q}$-factorial. Let $D\subseteq Y$ be the codimension $1$ part of the discriminant of $f$. Then $X$ is  Fano if and only if $-K_Y-\frac{1}{2}D$ is  $\mathbb{Q}$-ample. 
\end{lemma}

\begin{proof}
Since $Y$ is $\mathbb{Q}$-factorial,   $-K_{X}$ is $\mathbb{Q}$-linearly equivalent to $ -f^*(K_{Y}+\frac{1}{2}D)$. Since $f$ is finite, we obtain that $-K_X$ is $\mathbb{Q}$-ample if and only if $-K_Y-\frac{1}{2}D$ is  $\mathbb{Q}$-ample. 
\end{proof}

\section{Construction of Fano varieties with $\mathrm{Nef}(X)=\mathrm{Psef}(X)$ and $\rho(X)=\mathrm{dim}\, X-1$}

In this section, we will provide two kinds of  methods  to construct Fano varieties with $\mathrm{Nef}(X)=\mathrm{Psef}(X)$ and $\rho(X)=\mathrm{dim}\, X-1$.

\begin{const}
\label{exmp-112} 
We will first give two methods to construct  Fano varieties $X$ of dimension $n\geqslant 3$   which are   finite covers of the product  $M=(\p^1)^{n-2}\times \p^2$.

\begin{enumerate}
\item[1.] Let $X$ be a double cover  over $M$, branched along a  prime   divisor of degree $(2,...,2,k)$, where $k\in \{2,4\}$.

\item[2.] We assume that $n \geqslant 4$. We can write $M= (\p^1)^r \times \p^2 \times (\p^1)^{s}$ such that $r,s>0$ and $r+s=n-2$. Let $X_{r,2}$ be a double cover of $(\p^1)^r \times \p^2$, branched along a  prime divisor $D_1$ of degree $(2,...,2,2)$. 
Set $M_1=X_{r,2}\times (\p^1)^{s}$.  Let $X$ be a double cover of $M_1$, branched along a  prime  divisor $D$ which is equivalent to the pullback of some prime divisor  of degree $$(\underbrace{0,...,0}_{r},2,\underbrace{2,...,2}_{s})$$ \textit{via} 
$X_{r,2}\times (\p^1)^s\to (\p^1)^r \times \p^2\times (\p^1)^{s}$.   We note that $X$ is indeed isomorphic to  the following fiber product 

\centerline{
\xymatrix{
 X\cong M_1\times_M M_2 \ar[d] \ar[r]  &M_2 \ar[d]^{p_2} \\
M_1 \ar[r]_{p_1} &  *+[r]{(\p^1)^r\times \p^2\times (\p^1)^s}
}
}
\noindent where $p_1$ is a double cover branched along a prime divisor $D_1\times (\p^1)^s$ of degree  $$(\underbrace{2,...,2}_{r},2,\underbrace{0,...,0}_{s}),$$ and $p_2$ is a double cover branched along a prime divisor $(\p^1)^r \times D_2$ of degree  $$(\underbrace{0,...,0}_{r},2,\underbrace{2,...,2}_{s}).$$

\end{enumerate}
\end{const}

We will check that the constructions above do produce Fano varieties with $\mathrm{Nef}(X)=\mathrm{Psef}(X)$ and $\rho(X)=\mathrm{dim}\, X-1$. We will first prove two lemmas.

\begin{lemma}
\label{lem-relative-Pic}
Let $f \colon X\to Y$ be a  Fano fibration of relative dimension $1$ between normal projective varieties with $\mathbb{Q}$-factorial klt singularities. Then the relative Picard number of $f$ is $1$ if and only if $f^*D$ is irreducible for every prime divisor $D$ in $Y$.
\end{lemma}

\begin{proof}
If the relative Picard number of $f$ is $1$, then $f^*D$ is irreducible for every prime divisor $D$ in $Y$. Assume that the relative Picard number of $f$ is at least $2$. Since $f$ is a Fano fibration and $X$ has $\mathbb{Q}$-factorial klt singularities, we can run an MMP for $X$ over $Y$ by \cite[Cor. 1.3.3]{BCHM10}. We obtain the following sequence of birational maps over $Y$

\centerline{
\xymatrix{
X_0\ar@{-->}[r]  \ar[d]   &X_1  \ar[ld] \ar@{-->}[r]  & \cdots \ar@{-->}[r] & X_k\ar[dlll]^{f_k}   \\
Y
}
}

\noindent such that $X=X_0$ and that $X_k$ is a Mori fiber space. Let $\pi\colon  X_k\to B$ be the Mori fibration. Then there is a natural morphism $p\colon B\to Y$.  Since $f$ has relative dimension $1$, the natural morphism $B\to Y$ is birational.  

First we assume that $p$ contracts a divisor. Then there is a prime divisor $D$ in $Y$ such that $p^*D$ is reducible. Thus $f_k^*D$ is reducible , where $f_k\colon X_k\to Y$ is the natural morphism. Since the birational map $X_k \dashrightarrow X$ does not contract any divisor, we obtain that $f^*D$ is reducible.

Now we assume that  $p \colon B\to Y$ does not contract any divisor. Since $Y$ is $\mathbb{Q}$-factorial, the exceptional locus of $B\to Y$ is either empty or pure of codimension $1$. Hence $B\to Y$ is an isomorphism. As a consequence, the Picard number of  $X_k$ is one more than the one of $Y$.  Since the relative Picard number of $X$ over $Y$ is larger than $1$, there must be some $X_i$ in the  sequence above such that $X_i\to X_{i+1}$ is a divisorial contraction. In particular, if $f_i \colon X_i\to Y$ is the natural fibration, then there is a prime divisor $D$ in $Y$ such that $f_i^*D$ is reducible. Hence $f^*D$ is also reducible.  This completes the proof of the lemma.
\end{proof}

\begin{lemma}
\label{lem-fiber-product-picard-number}
Let $g \colon Y\to Z$ be an equidimensional  fibration between $\mathbb{Q}$-factorial klt projective varieties. Let $r \colon V\to Z$ be a Mori fibration of relative dimension $1$, where $V$ is a $\mathbb{Q}$-factorial klt variety.  Let $X$ be the fiber product $Y\times_{Z} V$. Let $f \colon X\to Y$ and $\pi \colon X\to V$  be the natural fibrations. 

\centerline{
\xymatrix{
X\ar[d]_f \ar[r]^{\pi}  &V \ar[d]^r & \\
Y \ar[r]_g & Z\\
}
}

\noindent Assume that  $X$ is a  $\mathbb{Q}$-factorial klt  Fano variety and that the discriminants of $g$ and $r$ do not have common components. Then the relative Picard number of   $f$ is $1$.
\end{lemma}

\begin{proof}
Assume the opposite. Then by Lemma \ref{lem-relative-Pic}, there is a prime divisor $D$ in $Y$ such that $f^*D$ is reducible. Let $R=f^{-1}(D)$ and let $R_1, R_2$  be two different components of $R$. Let $\alpha$ be a general point in $D$. Let $E= g(D)$ and let $\beta = g(\alpha)\in E$. Then $E$ is irreducible.  Since $g$ is equidimensional, either $E$ is a divisor or $E=Z$.  Since $X= Y\times_Z V$, the fiber $C'$ of $r$ over $\beta$ is isomorphic to the fiber $C$ of $f$ over $\alpha$. Thus $C'$ is reducible for $C$ is reducible. Since $\beta$ is a general point in $E$, we obtain that $E$ is a divisor and is contained in the discriminant of $r$.

Let  $R'=r^{-1}(E)$. Then, by Lemma \ref{lem-relative-Pic}, it is irreducible since $r$ is a Mori fibration.  We have the following commutative diagram

\centerline{
\xymatrix{
R\ar[d]_{f|_R} \ar[r]^{\pi|_R}  &R' \ar[d]^{r|_{R'}} & \\
D \ar[r]_{g|_D} & E\\
}
}

Since $X=Y\times _Z V$ and $g$ is equidimensional, the fibration $\pi$ is equidimensional. Since $R'$ is irreducible, we obtain that the  projections $\pi|_{R_1} \colon  R_1\to R'$ and $\pi|_{R_2} \colon  R_2\to R'$ are  surjective. In particular, general fibers of $\pi|_R$ are reducible. 

Let $a$ be a general point of $R'$ and let $b=r(a)\in E$. Since $X=Y\times _Z V$, the fiber of $\pi|_R$ over $a$ is the same as the fiber of $g|_D$ over $b$.  Since the discriminants of $g$ and $r$ do not have common components and $b$ is a general point in $E$, the fiber of $g|D$ over $b$ is irreducible. Hence the fiber of $\pi|_R$ over $a$ is irreducible. This is a contradiction.
\end{proof}

The verification of Construction \ref{exmp-112} is as follows.

\begin{lemma}
\label{lem-verifi-1}
Let $X$ be a variety constructed by the construction above. Assume that the branched loci  are in general positions. Then $X$ is a Fano variety with $\mathrm{Nef}(X)=\mathrm{Psef}(X)$ and $\rho(X)=\mathrm{dim}\, X-1$.
\end{lemma}

\begin{proof}
We will only prove the case of Construction \ref{exmp-112}.2.  The verification for the first case is similar and simpler. From the ramification formula, we see that $X$ is a Fano variety. We will now prove that $\rho(X)=\mathrm{dim}\, X-1$. By \cite[Lem. 3.4]{Dru16}, the Picard number of $X_{r,2}$ is $r+1$.  We write $$M_1=X_{r,2}\times (\p^1)^s = M_1'\times  \p^1 .$$  Then the natural fibration $g_1$ from the   $M_1'$ to $(\p^1)^{r-1}\times \p^2\times (\p^1)^{s-1}$ is a Mori fibration.  Let $M_2$ be as in Construction \ref{exmp-112}.2. We write $M_2=\p^1\times M_2'$. Then, by the same argument, the morphism $g_2$ from $ M_2'$ to   $(\p^1)^{r-1}\times \p^2\times (\p^1)^{s-1}$ is a Mori fibration. We note that $X$ is then  isomorphic to the following fiber product  

\centerline{
\xymatrix{
 X\cong M_1\times_{M'} M'_2 \ar[d] \ar[r]  &M'_2 \ar[d]^{g_2} \\
M_1'    \ar[r]_{g_1} &  *+[r]{M'=(\p^1)^{r-1}\times \p^2\times (\p^1)^{s-1}}
}
}

\noindent  Since the branched loci of the morphisms from $M_1$  and $M_2$ to $(\p^1)^{s}\times \p^2\times (\p^1)^{r}$ are in general positions, we may assume that  the discriminants of $g_1$ and  $g_2$ do not have common components. Moreover, we can assume that  $M_1$,  $M_2$ and $X$ are smooth. By  Lemma \ref{lem-fiber-product-picard-number}, the morphism $X\to M_2'$ also has relative Picard rank one. Therefore,  we have $\rho(X)=n-1$.

Finally, the property $\mathrm{Nef}(X)=\mathrm{Psef}(X)$ follows from the following lemma.
\end{proof}

\begin{lemma}
\label{lem-verification-nef=psef}
Let $f\colon X\to Y$ be a finite surjective morphism between $\mathbb{Q}$-factorial projective varieties. If $\rho(X)=\rho(Y)$ and if $\mathrm{Nef}(Y)=\mathrm{Psef}(Y)$, then $\mathrm{Nef}(X)=\mathrm{Psef}(X)$.
\end{lemma}

\begin{proof}
Since $\rho(X)=\rho(Y)$ and since $f$ is finite surjective, the pullback $f^* \colon \mathrm{N}^1(Y) \to \mathrm{N}^1(X) $ is an isomorphism. 
Thus if $D_X$ is a pseudoeffective $\mathbb{Q}$-divisor in $X$, then there is a $\mathbb{Q}$-divisor $D_Y$ in $Y$ such that $f^*D_Y \equiv D_X$. By the projection formula, $D_Y$ must be pseudoeffective. Thus it is nef, and then so is $D_X$. This completes the proof of the lemma.
\end{proof}

\begin{const}
\label{exmp-122} 
Let $B_1\cong B_2\cong \p^2$. We will give three examples of Fano varieties $X$ of dimension $n\geqslant 3$ which are  finite covers of the product $(\p^1)^{n-3}\times  W$, of degree $2$ or $4$, where $W$ is a normal ample hypersurface in $B_1\times B_2$.

\begin{enumerate}
\item[1.] Let $W$ be a normal divisor of degree $(1,1)$ in $B_1\times B_2$. Let $X$ be a double cover of $(\p^1)^{n-3} \times W$,  branched along a  prime   divisor $D$, which is  equivalent to the pullback of some prime divisor of degree $(2,...,2,2,2)$ \textit{via} $(\p^1)^{n-3} \times W \to (\p^1)^{n-3} \times  B_1\times B_2$.

\item[2.] Assume that $n\geqslant 4$. Let $ Y_{n-3,2} $ be a double cover of $(\p^1)^{n-3}\times B_1$, branched along a   prime divisor $D_1$ of degree $(2,..,2,2)$. Let $Y=Y_{n-3,2}\times B_2$ and let $p$ be the natural projection from $Y$ to   $B_1\times B_2$. Let $W$ be a normal hypersurface of degree $(1,k)$ in $B_1\times B_2$, where $k\in\{1,2\}$. Let $X=p^*W$.   Equivalently, $X$ is isomorphic to  the following fiber product 

\centerline{
\xymatrix{
X \cong ((\p^1)^{n-3} \times W)\times_S Y\ar[d] \ar[r]  &*+[r]{Y=Y_{n-3,2}\times B_2} \ar[d]^{g} \\
(\p^1)^{n-3} \times W \ar[r]_{j} & *+[r]{S= (\p^1)^{n-3}\times B_1\times B_2}
}
}
\noindent where $j$ is the natural injection and $g$ is a double cover branched along a prime divisor $D_1\times B_2$ of degree $(2,...,2,2,0)$.

\item[3.] Assume that $n\geqslant 5$. Let   $r,s>0$ be two integers such that $r+s=n-3$. Let $Y_{r,2}$ be a double cover of $(\p^1)^{r}\times  B_1$, branched along a   prime divisor $D_1$ of degree $(2,...,2,2)$. Let $Y_{s,2}$ be a double cover of $(\p^1)^{s} \times B_2$, branched along a  prime divisor of degree $(2,...,2,2)$. Let $Y=Y_{r,2}\times Y_{s,2}$ and let $p$ be the projection  from $Y$  to $B_1\times B_2$. Let $W$ be a normal hypersurface of degree $(1,1)$ in $B_1\times B_2$. We set $X=p^*W$.  Equivalently, $X$ is isomorphic to  the following fiber product 

\centerline{
\xymatrix{
X\cong ((\p^1)^{r} \times W \times (\p^1)^{s})\times_S Y\ar[d] \ar[r]  & *+[r]{Y=Y_{r,2}\times Y_{s,2}} \ar[d]\\
(\p^1)^{r} \times W \times (\p^1)^{s} \ar[r]_{} & *+[r]{S= (\p^1)^{r}\times B_1\times B_2 \times (\p^1)^{s}}
}
}

\end{enumerate}
\end{const}

The following lemma shows that the constructions above can produce Fano varieties $X$ which satisfy  $\mathrm{Nef}(X)=\mathrm{Psef}(X)$ and $\rho(X)=\mathrm{dim}\, X-1$.

\begin{lemma}
\label{lem-verifi-2}
Let  $X$ be a variety  obtained by  Construction \ref{exmp-122}. Assume that $W$ is smooth, and that the branch loci of the finite morphisms are in general positions. Then $X$ is a Fano variety with $\mathrm{Nef}(X)=\mathrm{Psef}(X)$ and $\rho(X)=\mathrm{dim}\, X-1$. 
\end{lemma}

\begin{proof}
We will only prove the case of Construction \ref{exmp-122}.3 The verification for the other cases is similar and simpler. From the ramification formula, we see that $X$ is a Fano variety. We will now prove that $\rho(X)=\mathrm{dim}\, X-1$.  With the notation of Construction \ref{exmp-122}.3,  we let $X_{r,W}$ be the following fiber product 

\centerline{
\xymatrix{
X_{r,W}\cong W \times_{B_1}  Y_{r,2} \ar[d] \ar[r]  & *+[r]{ Y_{r,2} } \ar[d]_{g_2}\\
  W \ar[r]_{g_1}  & *+[r]{ B_1 }
}
}

\noindent By Lefschetz theorem (see \cite[Example 3.1.25]{Laz04}), $W$ has Picard number $2$, and hence $W\to B_1$ is a Mori fibration. We may assume that the discriminants  of $g_1$ and $g_2$ do not have common components, and that $Y_{r,2}$ and $X_{r,W}$ are smooth.   Thus by  Lemma \ref{lem-fiber-product-picard-number}, the morphism $X_{r,W}\to Y_{2,r} $ also has relative Picard rank one. Moreover, by \cite[Lem. 3.4]{Dru16}, $\rho(Y_{r,2})=r+1$. Therefore, we have $\rho(X_{r,W})=r+2$. Indeed, we also note that $X_{r,W}$ is a double cover of $(\p^1)^{r}\times W$, branched along a prime divisor which is the pullback of a divisor of degree $(2,...,2,2,0)$ \textit{via} the morphism $(\p^1)^r\times W\to (\p^1)^r\times B_1\times B_2.$

Similarly, we let $X_{W,s}$ be the following product. 

\centerline{
\xymatrix{
X_{W,s}\cong W \times_{B_2}   Y_{s,2}  \ar[d] \ar[r]  & *+[r]{  Y_{s,2}} \ar[d]_{ }\\
  W \ar[r]_{ }  & *+[r]{   B_2}
}
}

\noindent Then $X_{W,s}$ is smooth with Picard number $s+2$.  It is in fact isomorphic to a double cover of $W\times (\p^1)^{s}$, branched along a prime divisor which is the pullback of a divisor of degree $(0,2,2,...,2 )$ \textit{via} the injection $ W \times  (\p^1)^{s} \to    B_1\times B_2 \times  (\p^1)^{s}.$ We remark that $X$ is then isomorphic to the following fiber product 

\centerline{
\xymatrix{
X \cong Z_1\times_{R} Z_2 \ar[d] \ar[r]  & *+[r]{Z_2= (\p^1)^{ {r-1} } \times X_{W,s}} \ar[d]_{ }\\
  Z_1=X_{r,W}\times (\p^1)^{s-1}  \ar[r]_{ } & *+[r]{R= (\p^1)^{r-1}\times W \times (\p^1)^{s-1} }
}
}

\noindent Both of the morphisms $Z_1\to R$ and $Z_2\to R$ above  are Mori fibrations. We may assume that their discriminants do not have common components. We may also assume further that $X$ is smooth. Then  by  using Lemma \ref{lem-fiber-product-picard-number} again, we obtain that $$\rho(X)=r+2+s=\mathrm{dim}\, X-1.$$ By  Lemma \ref{lem-verification-nef=psef}, we have   $\mathrm{Nef}(X)=\mathrm{Psef}(X)$. This completes the proof of the lemma.
\end{proof}


\section{Fibrations on  varieties in Theorem \ref{thm-classification}}

Let $X$ be a $\mathbb{Q}$-factorial klt Fano variety. Then for every face $V$ of $\NE(X)$, there is a  fibration $f \colon X\to Y$ which contracts exactly the curves whose classes are in $V$. Conversely, if $f \colon X\to Y$ is a  fibration, then the curves in the fibers of $f$ generates a face $V$ of $\NE(X)$ (See for example \cite[\S 3-2]{KMM87}). 
 
 \begin{lemma}
\label{lem-base-fano}
Let $X$ be a  Fano variety  with $\mathbb{Q}$-factorial klt (resp. locally factorial canonical) singularities  such that $\mathrm{Nef}(X)=\mathrm{Psef}(X)$. Let $f \colon X\to Y$ be any  fibration. Then $Y$ is also a  Fano variety with $\mathbb{Q}$-factorial klt (resp. locally factorial canonical) singularities  such that $\mathrm{Nef}(Y)=\mathrm{Psef}(Y)$. 
\end{lemma}

\begin{proof}
By \cite[Lem. 4.2]{Dru16}, we have $\mathrm{Nef}(Y)=\mathrm{Psef}(Y)$. If $X$ has locally factorial canonical  singularities, then so has $Y$ by \cite[Cor. 4.8]{Dru16}. If $X$ has $\mathbb{Q}$-factorial klt  singularities, then the same argument of the proof of \cite[Cor. 4.8]{Dru16} shows that $Y$ also has $\mathbb{Q}$-factorial klt  singularities. By \cite[Thm. 3.1]{FG12},   $-K_Y$ is big.  Since  $\mathrm{Nef}(Y)=\mathrm{Psef}(Y)$, this implies that  $-K_Y$ is ample.
\end{proof}

The objective of this section is to prove the following two  lemmas.

\begin{lemma}
\label{lem-base-*}
Let $X$ be a Fano variety  with locally factorial canonical singularities such that $X$ is smooth in codimension $2$, $\rho(X) = \mathrm{dim}\, X-1$ and  $\mathrm{Nef}(X) = \mathrm{Psef}(X)$. Assume that there is no fibration from $X$ to $\p^1$. Let $f \colon X\to Y$ be a  fibration. Then $Y$ is  also a Fano variety with locally factorial canonical singularities such that $Y$ is smooth in codimension $2$, $\rho(Y) = \mathrm{dim}\, Y-1$ and  $\mathrm{Nef}(Y) = \mathrm{Psef}(Y)$.
\end{lemma}

\begin{lemma}
\label{lem-base-*-p1}
Let $X$ be a Fano variety with locally factorial canonical singularities such that $X$ is smooth in codimension $2$, $\rho(X) = \mathrm{dim}\, X-1$ and  $\mathrm{Nef}(X) = \mathrm{Psef}(X)$. Assume that there is a  fibration $f_1 \colon X\to W$ and a projective  morphism $f_2 \colon X\to (\p^1)^r$   such that the product $f_1\times f_2 \colon X\to W\times (\p^1)^r$ is a finite surjective morphism. Then $W$ is  also a   Fano variety with locally factorial canonical singularities such that $W$ is smooth in codimension  $2$, $\rho(W) = \mathrm{dim}\, W-1$ and  $\mathrm{Nef}(W) = \mathrm{Psef}(W)$.
\end{lemma}

We will first prove some preliminary results.

\begin{lemma}
\label{lem-relative-1-Mori}
Let $X$ be a $\mathbb{Q}$-factorial klt Fano variety such that $\mathrm{Nef}(X)=\mathrm{Psef}(X)$. Let $f \colon X\to Y$ be a  fibration. If the relative dimension of $f$ is $1$, then $f$ is a Mori fibration.
\end{lemma}

\begin{proof}
There is a Mori fibration $X\to Z$ over $Y$ by \cite[Lem. 4.4]{Dru16}. The variety $Z$ is a $\mathbb{Q}$-factorial klt Fano variety such that $\mathrm{Nef}(Z)=\mathrm{Psef}(Z)$ by Lemma \ref{lem-base-fano}. Since $f$ has relative dimension $1$, the natural fibration $Z\to Y$ is birational. Since $\mathrm{Nef}(Z)=\mathrm{Psef}(Z)$, we have  $Z\cong Y$ by \cite[Lem. 4.4]{Dru16}. This implies that $f$ is a Mori fibration.
\end{proof}

\begin{lemma}
\label{lem-Mori-fibration-relative-1}
Let $f \colon X\to Y$ be a Mori fibration such that $X$ has locally factorial canonical singularities. If $X$ is smooth in codimension $k$ and $f$ is equidimensional of relative dimension $1$, then $Y$ is also smooth in codimension $k$.  
\end{lemma}

\begin{proof}
Let $y\in Y$ be a point. The problem is local around $y$. Without loss of generality, we may assume that $Y$ is   affine. Since $f$ is a Fano fibration of relative dimension $1$, by the main theorem of \cite{AW93}, the linear system $|-K_X|$ is basepoint-free. Let $Z$ be a general divisor in the linear system $|-K_X|$. Then we can assume that $Z$ does not contain the singular locus $X_{sing}$ of $X$. In particular, the codimension of $Z\cap X_{sing}$   in $Z$ is at least $k+1$. Moreover, by Bertini theorem, we can choose $Z$ so that $Z\backslash X_{sing}$ is smooth. Thus  $Z$ is smooth in codimension $k$.  We can also assume that $Z$ does not contain  any component of the fiber of $f^{-1}(y)$.  Since $f$ is equidimensional,  by shrinking $Y$ if necessary, we may then assume further that  the morphism $f|_Z \colon Z\to Y$ is  finite surjective. We note that the degree of $f|_Z$ is $2$, for the intersection number of $-K_X$ and a fiber of $f$ is $2$. 

We note that $Y$ has klt singularities (see \cite[Cor. 3.5]{Fuj99}). Thus it is  Cohen-Macaulay. Moreover, it is locally factorial (see for example \cite[Lem. 4.6]{Dru16}). Since $Z$ is smooth in codimension $k$, by Lemma \ref{lem-double-cover}, $Y$ is smooth in codimension $k$.
\end{proof}

Remark that this lemma is not true without assuming that $X$ is locally factorial. There is a Mori fibration $f \colon X\to Y$ such that $X$ is a threefold with isolated terminal singularities and  that $Y$ is a singular surface, see \cite[Example-Definition 1.1.1]{MoP08}

\begin{lemma}
\label{lem-base-Picard-number}
Let $X$ be a Fano variety with $\mathbb{Q}$-factorial klt singularities such that $\rho(X)=\mathrm{dim}\, X-1$ and $\mathrm{Nef}(X)=\mathrm{Psef}(X)$. Let $f \colon X\to Y$ be a Mori fibration. Then  the dimension of  $Y$ is equal to $\rho(Y)$ or $\rho(Y)+1$.
\end{lemma}

\begin{proof}
By  Lemma \ref{lem-base-fano}, $Y$ is a $\mathbb{Q}$-factorial klt Fano variety with $\mathrm{Nef}(Y)=\mathrm{Psef}(Y)$. In particular, we have $\rho(Y)\leqslant \mathrm{dim}\, Y$ by \cite[Lem. 4.9]{Dru16}. Since $f$ is a Mori fibration, we have $\rho(Y)=\rho(X)-1$. Hence $\mathrm{dim}\, Y \leqslant \mathrm{dim}\, X-1=\rho(X)=\rho(Y)+1$. Thus $\mathrm{dim}\, Y$ is equal to $\rho(Y)$ or $\rho(Y)+1$.
\end{proof}



\begin{lemma}
\label{lem-rho=dim-to-p1}
Let $X$ be a $\mathbb{Q}$-factorial klt Fano variety such that $\mathrm{Nef}(X)=\mathrm{Psef}(X)$ and $\rho(X)=\mathrm{dim}\, X$. Then there is a  fibration from $X$ to $\p^1$.
\end{lemma}

\begin{proof}
Let $V$ be a face in $\NE(X)$ of codimension $1$ and let $f \colon X\to B$ be the corresponding contraction. Then $\mathrm{dim}\, B= \rho(B)=1$ by  \cite[Lem. 4.9.2]{Dru16}. This shows that $B\cong \p^1$.
\end{proof}

\begin{lemma}
\label{lem-equidimensional}
Let $X$ be a Fano variety with $\mathbb{Q}$-factorial klt singularities such that $\rho(X)=\mathrm{dim}\, X-1\geqslant 2$ and $\mathrm{Nef}(X)=\mathrm{Psef}(X)$. Assume that there is no fibration from $X$ to $\p^1$. Then   every Mori fibration $f \colon X\to Y$  is equidimensional of relative dimension $1$.
\end{lemma}

\begin{proof}
We will prove the lemma by induction on the dimension of $X$.   If $\mathrm{dim}\, X=3$, then $\mathrm{dim}\, Y=2$.  In this case, the Mori fibration $f$ must be equidimensional by Lemma \ref{lem-relative-Pic}. Assume that the lemma is true if $\mathrm{dim}\, X = n$ for some $n\geqslant 3$. 
 
Now we assume that $\mathrm{dim}\, X=n+1$. Assume that $f$ is the fibration which corresponds to an extremal ray $R$ of  $\NE(X)$. Let $R_1$ be another extremal ray such that $R$ and $R_1$ generates a face $V$ of dimension $2$. Let $g \colon X\to Z$ be the Mori fibration which corresponds to $R_1$ and let $X\to W$ be the  fibration corresponding to $V$. We obtain the following commutative diagram 

\centerline{
\xymatrix{
X  \ar[r]^f \ar[d]^{g}
&Y  \ar[d]^h
\\
Z   \ar[r]^p   
&W  
}
}

By Lemma \ref{lem-base-fano}, $Z$ is a  $\mathbb{Q}$-factorial klt Fano variety such that $\mathrm{Nef}(Z)=\mathrm{Psef}(Z)$. By Lemma \ref{lem-base-Picard-number}, $\rho(Z)$ is either $\mathrm{dim}\, Z$ or $\mathrm{dim}\, Z-1$. Since there is no fibration from $Z$ to $\p^1$, Lemma \ref{lem-rho=dim-to-p1} shows that $\rho(Z)=\mathrm{dim}\, Z-1$. Since $\rho(Z)=\rho(X)-1=n-1$, we obtain that $\mathrm{dim}\, Z=n$.  By induction hypothesis, the morphism $p$ is equidimensional of relative dimension $1$.  Suppose that there is a fiber $F$ of $f$ such that $\mathrm{dim}\, F\geqslant 2$.  Since none of the curves in $F$ is   contracted by $g$, the dimension of $E=g(F)$ is also larger than $1$. Moreover, since $h(f(F))$ is a point, $p(E)$ is also a point. We obtain a contradiction. Hence $f$ is equidimensional of relative dimension $1$. This completes the induction and the proof of the lemma.
\end{proof}

Now we will prove Lemma \ref{lem-base-*} and Lemma \ref{lem-base-*-p1}.

\begin{proof}[{Proof of Lemma \ref{lem-base-*}}]
We will prove the lemma by induction on the relative Picard number $k$ of $f$. If $f$ is a Mori fibration, then by Lemma \ref{lem-equidimensional}, $f$ is equidimensional of relative dimension $1$. Hence $\rho(Y)=\mathrm{dim}(Y)-1$. Moreover, $Y$ is a Fano variety with locally factorial canonical singularities such that $\mathrm{Nef}(Y)=\mathrm{Psef}(Y)$ by Lemma \ref{lem-base-fano} and is smooth in codimension $2$ by Lemma \ref{lem-Mori-fibration-relative-1}.

Assume that the lemma is true for $k\leqslant l$ for some $l\geqslant 1$. Now we consider the case when $k=l+1$. There is a Mori fibration $g \colon X\to Z$ over $Y$ since $\mathrm{Nef}(X)=\mathrm{Psef}(X)$ (see \cite[Lem. 4.4]{Dru16}). Hence, by induction, $Z$ is a Fano variety with locally factorial canonical singularities such that $Z$ is smooth in codimension $2$, $\rho(Z) = \mathrm{dim}\, Z-1$ and  $\mathrm{Nef}(Z) = \mathrm{Psef}(Z)$. Moreover, the relative Picard number of $Z\to Y$ is $l$. Thus, by induction hypotheses, $Y$ is a Fano variety with locally factorial canonical singularities such that $Y$ is smooth in codimension $2$, $\rho(Y) = \mathrm{dim}\, Y-1$ and  $\mathrm{Nef}(Y) = \mathrm{Psef}(Y)$. This completes the proof of the Lemma.
\end{proof}

\begin{proof}[{Proof of Lemma \ref{lem-base-*-p1}}]
We will prove the lemma by induction on $r$. If $r=1$, then the morphism $f_1$ is equidimensional of relative dimension $1$ since $f_1\times f_2$ is finite. By Lemma \ref{lem-relative-1-Mori}, $f_1$ is a Mori fibration. By Lemma \ref{lem-base-fano} and Lemma \ref{lem-Mori-fibration-relative-1}, we obtain that $W$ is a  Fano variety  with locally factorial canonical singularities such that $W$ is smooth in codimension $2$, $\rho(W) = \mathrm{dim}\, W-1$ and  $\mathrm{Nef}(W) = \mathrm{Psef}(W)$. Assume that the lemma is true for  $r\leqslant k$ where $k\geqslant 1$ is an integer.

Now we assume that $r=k+1$.   We write $(\p^1)^{k+1}=(\p^1)^k\times \p^1$. Let  $q \colon (\p^1)^{k+1}\to (\p^1)^k$ be the natural projection onto the product of the first $k$ $\p^1$'s and  let  $p \colon (\p^1)^{k+1}\to \p^1$ be the natural projection onto the last factor.  Let $\pi \colon  X\to Y$ be the Stein factorization  of the product $f_1\times  q \colon X\to W\times (\p^1)^k$. Then the product $\pi \times p \colon  X\to Y\times \p^1$  is also a finite surjective morphism. Hence, by induction hypotheses, $Y$ is a Fano variety with locally factorial canonical singularities such that $Y$ is smooth in codimension $2$, $\rho(Y) = \mathrm{dim}\, Y-1$ and  $\mathrm{Nef}(Y) = \mathrm{Psef}(Y)$. Let $g_1 \colon Y\to W$ and  $g_2 \colon Y\to (\p^1)^k$ be the natural projections. The morphism $g_1$ is a fibration, for the composition $f_1=g_1\circ \pi \colon X\to W$ is. Since $X\to W\times (\p^1)^k\times \p^1$ is finite surjective, we obtain that $g_1\times g_2 \colon Y\to W\times (\p^1)^k$ is also finite surjective. By induction hypotheses, $W$ is a Fano variety with locally factorial canonical singularities such that $W$ is smooth in codimension $2$, $\rho(W) = \mathrm{dim}\, W-1$ and  $\mathrm{Nef}(W) = \mathrm{Psef}(W)$. This completes the proof of the lemma.

\centerline{
\xymatrix{
X \ar[rrd]_{\pi} \ar[rr]^{\mathrm{finite}}_{\mathrm{surjective}}  & & Y\times \p^1 \ar[d] \ar[rr]^{\mathrm{finite\ \ }}_{\mathrm{surjective\ }}   & & W\times {(\p^1)^{k+1}} \\
& & Y \ar[rr]^{\mathrm{finite}}_{\mathrm{surjective}}  & & W\times {(\p^1)^k} & &
}
}
\end{proof}

\section{Finite morphisms between Fano threefolds}
\label{Finite morphisms between Fano varieties}

In this section, we will prove some results on finite morphisms between Fano threefolds.  Particularly, we will  classify  these morphisms under some assumptions.  As a corollary, we will prove Theorem \ref{thm-classification-3f} by using the classification results  of Proposition \ref{prop12-finite-over-12} and Proposition \ref{prop12-finite-over-22}. Moreover, in the last two sections of the paper, by using the results of this section,   we can  show  that the factor $X_2$ in Theorem \ref{thm-classification} is a finite cover of degree at most $4$ over $(\p^1)^s\times \p^2$ or over $(\p^1)^s \times W$, where $W$ is an ample hypersurface in $\p^2\times \p^2$. It is the crucial ingredient for the proof of Theorem \ref{thm-classification}.

We recall that if $X$ is a Fano threefold with Gorenstein canonical singularities, then $$h^0(X,\sO_X(-K_X))=-\frac{1}{2}K_X^3+3.$$  In particular, $-K_X^3$ is a positive even integer (see \cite[\S 4.4]{Reid83}). We will first prove some lemmas. 

\begin{lemma}
\label{lem-smooth-Fano-surface-curve}
Let $f \colon X\to Y$ be a  Fano fibration of relative dimension $1$ between normal quasi-projective varieties. Assume that $X$ is smooth in codimension $2$. Then there is an open subset $U$ of $Y$ whose complement has codimension at least $2$ such that $f$ has reduced fibers over $U$.
\end{lemma}

\begin{proof}
By taking general hyperplane sections in $Y$, we can reduced to the case when $X$ is a smooth surface and $Y$ is a smooth curve.

Let $y$ be a point in $Y$ and let $E=f^*y$. Then $-K_X\cdot E=2$. Since $K_X$ is Cartier and $f$ is a Fano fibration, this implies that $E$ has at most two components.   If $E$ has two components $C$ and $D$, then both of them are reduced.  If $E$ has one component $C$, then $E$ is also reduced since there is a section from $Y$ to $X$ by Tsen's theorem.
\end{proof}

\begin{lemma}
\label{lem-normal-divisor}
Let $X$ be a normal threefold and let $B_1\cong B_2\cong \p^2$. Assume that there are two equidimensional fibrations $f_1 \colon X\to B_1$ and $f_2 \colon X\to B_2$ such that $f_1\times f_2 \colon X\to B_1\times B_2$ is finite onto its image $W$. Assume that $W$ is of degree $(p,q)$ such that $p,q\in \{1,2\}$. Then $W$ is normal.
\end{lemma}

\begin{proof}
Since $W$ is Cohen-Macaulay, we only need to prove that it is smooth in codimension $1$. Let $Z$ be the singular locus of $W$.  Let $p_i \colon W\to B_i$ be the natural projection for $i=1,2$. Then $p_1$ is equidimensional. Since $W$ is Cohen-Macaulay and $B_1$ is smooth, this implies that $p_1$ is flat.

Since general fibers of $f_1 \colon X\to B_1$ are irreducible, general fibers of $p_1$ are  irreducible.   General fibers of $p_1$ are also reduced for $W$ is reduced. Since $q\in\{1,2\}$, we obtain  that general fibers of $p_1$ are smooth rational curves. Since $p_1$ is flat, this implies that $p_1(Z)$ is a proper subvariety of $B_1$.   By symmetry, we can obtain that $p_2(Z)$ is also a proper subvariety of $B_2$.  Moreover, this argument also shows that if $p=1$ or $q=1$, then $W$ is smooth.

We now assume that  $p=q=2$.  In this case,  every fiber of $p_1$ is a conic in $B_2$, and we have an induced morphism  from $B_1$ to the Chow variety of conics in $B_2$.   Assume by contradiction that there  is  an irreducible component $E$ of $Z$ such that $E$ is a surface.  We note that every reducible fibers of $p_1$  is the union of two smooth rational curves. Since $p_1$ is flat,  this implies that there is a prime divisor $D_1$ in $B_1$ such that $p_1$ has double fibers over $D_1$ and that $E$ is the support of $p_1^* D_1$.   Let  $D_2=p_2(E)$, which must be a prime divisor  in $B_2$.  Since $p_1$ is a conic fibration, this implies that $D_2$ is a line, and every  fiber of $p_1$ over $D_1$ is the double line with support $D_2$.  As a consequence, the induced morphism from $B_1$ to the Chow variety of conics in $B_2$ contracts the divisor $D_1$. Since $B_1$ has Picard number $1$, this implies that every fiber of $p_1$ is the double line with support $ D_2$. We obtain a contradiction.   Hence $Z$ has codimension at least $2$ and  $W$ is normal.
\end{proof}

\begin{lemma}
\label{lem-normal-hypersurface-equidim}
Let $W$ be an ample normal hypersurface in $\p^2\times \p^2$. Assume that $W$ is a Fano variety with $\mathbb{Q}$-factorial klt singularities. Then the natural morphisms from $W$ to $\p^2$ are equidimensional.
\end{lemma}

\begin{proof}
Since $W$ is ample, it has Picard number $2$ by the Lefschetz theorem (see \cite[Example 3.1.25]{Laz04}).  Thus, if $g \colon W\to \p^2$ is a natural projection, then it is a Mori fibration. Since $W$ is a $\mathbb{Q}$-factorial klt threefold, this implies that  $g$ is equidimensional by Lemma \ref{lem-relative-Pic}.
\end{proof}

\begin{lemma} 
\label{lem--K-big}
Let $f \colon X\to Y$ be a finite surjective morphism  between normal $\mathbb{Q}$-Gorenstein varieties. Assume that  $-K_X$ is big. Then $-K_Y$ is a big divisor.
\end{lemma}

\begin{proof}
See \cite[Lem. 8.1]{Dru16}.
\end{proof}

\begin{lemma}
\label{lem-Gorenstein-Fano-fibration}
Let $f \colon X\to Y$ be an equidimensional Fano fibration of relative dimension $1$ such that $X$ has Gorenstein canonical singularities and $Y$ is smooth. Then every fiber of $f$ has at most two irreducible components. Moreover, every reducible fiber is reduced and is the union of two smooth rational curves which meets at one point.
\end{lemma}

\begin{proof}
Since $-K_X$ is relatively ample, $R^if_*(\sO_X(-K_X))=0$ for all $i>0$ by the Kawamata-Viehweg vanishing theorem (see \cite[Thm. 1-2-5]{KMM87}). Since $f$ is equidimensional of relative dimension $1$, for every $y\in Y$, we have $h^2(X_y, \sO_{X_y}(-K_X|_{X_y}))=0$. By   \cite[Thm. III.12.11]{Har77}, we obtain that, for all $y\in Y$, $$R^1f_*(\sO_X(-K_X)) \otimes k(y)\to  H^1(X_y, \sO_{X_y}(-K_X|_{X_y}))$$ is an isomorphism, where $k(y)$ is the residue field of $y$. Hence $$h^1(X_y, \sO_{X_y}(-K_X|_{X_y}))=0$$ for all $y\in Y$. By \cite[Thm. III.12.11]{Har77} again, we obtain that $$ f_*(\sO_X(-K_X)) \otimes k(y)\to  H^0(X_y, \sO_{X_y}(-K_X|_{X_y}))$$ is an isomorphism for all $y\in Y$.

We note that general fibers of $f$ are smooth rational curves. Hence the Euler characteristic of the restriction of $\sO_X(-K_X)$ on a general fiber of $f$ is $3$. Since $f$ is equidimensional, $X$ is Cohen-Macaulay and $Y$ is smooth, the morphism $f$ is flat. Since $\sO_X(-K_X)$ is locally free on $X$, it is flat over $Y$. Thus the Euler characteristic of the restriction of $\sO_X(-K_X)$ on every fiber of $f$ is $3$ (see the proof of \cite[Thm. III.9.9]{Har77}). This shows that $f_*(\sO_X(-K_X))$ is a locally free sheaf of rank $3$ by \cite[Cor. III.12.9]{Har77}.

The problem is local on $Y$. Thus, we may assume that $Y$ is affine and $f_*(\sO_X(-K_X))$ is a free sheaf of rank $3$. By the main theorem of \cite{AW93}, $\sO_X(-K_X)$ is $f$-relatively generated. Since the sheaf $f_*(\sO_X(-K_X))$ has rank $3$,  $f^*f_*(\sO_X(-K_X))$ induces a morphism, over $Y$, $$\phi \colon  X\to \p^2\times Y.$$ Since $f_*(\sO_X(-K_X)) \otimes k(y)\to  H^0(X_y, \sO_{X_y}(-K_X|_{X_y}))$ is an isomorphism and  $-K_X|_{X_y}$ is very ample for general $y\in Y$, the morphism $\phi$ is birational onto its image $X'$. Since $-K_X$ is $f$-relatively ample, we obtain that  $X\to X'$ is the normalization morphism.

Since $-K_X$ has degree $2$ on every fiber of $f$, we obtain that $X'_y  \subseteq \p^2\times \{y\}$ is a conic  for every $y\in Y$. Since $X\to X'$ is just the normalization,   the lemma then follows.
%
\end{proof}

The following lemma is the first classification result on finite morphisms in this section.

\begin{lemma}
\label{lem-p1p1p1p1-hypersurface}
Let $Z_1, Z_2$ be two surfaces isomorphic to $\p^1\times \p^1$. Let $Z$ be a normal hypersurface of degree $(1,1,1,1)$ in $Z_1\times Z_2$. Assume that there is a finite morphism $f \colon X\to Z$ of degree   $d\leqslant 2$ such that 

\begin{enumerate}
\item[-] $X$ is Fano threefold with   Gorenstein canonical singularities;
\item[-] the natural morphisms $f_1 \colon X\to Z_1$ and $f_2 \colon X\to Z_2$ induced by $f$ are  equidimensional  fibrations;
\item[-] $K_X\equiv \frac{1}{d} f^*K_Z$.
\end{enumerate} 

Then  $f$ is an isomorphism.
\end{lemma}

\begin{proof}
Assume that $f$ is of degree $2$.  Let $D$ be the codimension $1$ part of the discriminant of $f \colon X\to Z$. Then  $K_X$ is linearly equivalent to  $f^*(K_Z)+f^{-1}(D)$. Since $K_X$ and $K_Z$ are Cartier divisors, we obtain that $f^{-1}(D)$ is  Cartier. Since $f$ is finite, this implies that $D$ is $\mathbb{Q}$-Cartier.  Then $K_X$ is $\mathbb{Q}$-linearly equivalent to  $f^*(K_Z+\frac{1}{2}D)$. Thus  we have $D\equiv -K_Z$ (it is here we use the condition $K_X\equiv \frac{1}{d} f^*K_Z$).  In particular, $D$ is $\mathbb{Q}$-ample.

There are two natural morphisms from $Z$ to $\p^1\times Z_2$ induced by the natural projections from $Z_1\times Z_2$ to $\p^1\times Z_2$. Let $\phi \colon Z\to \p^1\times Z_2$ be one of them. Since $Z$ is of degree $(1,1,1,1)$, we obtain that $\phi$ is birational. Moreover, $\phi$ is not an isomorphism since   the Picard number of $Z$ is $4$ by the Lefschetz theorem (\cite[Example 3.1.25]{Laz04}). Thus its exceptional locus is non empty. Moreover, it is pure of codimension $1$ since  $\p^1\times Z_2$ is smooth. Hence the morphism $\phi$ contracts at least one divisor. Since the natural fibration $g_2 \colon Z\to Z_2$ factors through $\phi$, we obtain that there is a prime divisor $H$ in $Z_2$ such that $g_2^*H$ is reducible.

The fibration $g_2$ is equidimensional since $f_2$ is. Every fiber of $g_2$ is isomorphic to a divisor of degree $(1,1)$ in $\p^1\times \p^1$. Hence it is either a reduced fiber isomorphic to $\p^1$  or the union of two smooth rational curves which intersect at one point. This shows that $g_2^*H$ has exactly two components $R_1,R_2$ and every fiber of $g_2$ over $H$ has exactly two components. 

\centerline{
\xymatrix{
X  \ar[r]^f \ar[rd]_{f_2} &Z  \ar[d]_{g_2}^{\mathrm{every \ fiber\ is\  a \ divisor\ of\ degree\ }(1,1) \ \mathrm{in \ \p^1\times \p^1}} \\
&   Z_ 2
}
}

The divisor $f_2^*H$ has at least two components.  Since $X$ has Gorenstein canonical singularities and $f_2$ is equidimensional of relative dimension $1$, every fiber of $f_2$ has at most two components by Lemma \ref{lem-Gorenstein-Fano-fibration}.  We obtain that $f_2^*H$ also has exactly two components $S_1,S_2$.  Assume that $f(S_1)=R_1$ and $f(S_2)=R_2$. Let  $b$ be any point in $H$ and let $G$ be the fiber of $g_2$ over $b$. Then $G$ has exactly two irreducible components by the previous paragraph. We write $G=G_1\cup G_2$, where $G_1=G\cap R_1$ and $G_2=G\cap R_2$.  Let $F$ be the fiber of $f_2$ over $b$. Then $F$ has at least two components. By Lemma \ref{lem-Gorenstein-Fano-fibration} again,    the fiber $F$ is reduced and is the union of two  smooth rational curves $F_1$ and $F_2$  which meet at one point. We assume that $F_i\subseteq S_i$ for $i=1,2$. Since $F$ is reduced, we also obtain that $G$ is not contained in  the discriminant of $f$.

Both $f|_{F_1} \colon F_1\to G_1$ and $f|_{F_2} \colon F_2\to G_2$  are  finite morphisms of degree $2$, for $G$ is not contained in  the discriminant of $f$.  Since $G_1\backslash G_2$ is  simply connected,  it contains a point which is in the branch locus of $f|_{F_1} \colon F_1\to G_1$. Since $g_2$, being an equidimensional fibration from a Cohen-Macaulay variety to a smooth variety, is flat, $G_1\backslash G_2$ is contained in  $Z_{ns}$, the smooth locus of $Z$. Moreover, over  $Z_{ns}$, the morphism $f$ is branched exactly along   $Z_{ns}\cap D$ by  the Zariski purity theorem (see \cite[Prop. 2]{Zar58}). Thus $D$ meets $G_1\backslash G_2$ at at least one point. By symmetry,  $D$ meets $G_2\backslash G_1$ at at least one point.  However, since $D\equiv -K_Z$  and $g_2$ is a Fano fibration,  the intersection number   $D\cdot G$ is equal to $2$. Thus we obtain that $G$ and $D$ meet at exactly two points which are smooth points of $G$. 

Since $b$ is chosen arbitrarily, we obtain that $D\cap R_1\cap R_2$ is empty. However, since $R_1\cap R_2$ is a subscheme of dimension $1$ and $D$ is $\mathbb{Q}$-ample, $D\cap R_1\cap R_2$ is not empty. We obtain a contradiction.
\end{proof}

\subsection{Finite covers of $\p^1\times \p^2$} 

In this subsection, we will study finite morphisms from a Fano threefold $X$ to $\p^1\times \p^2$.

\begin{prop}
\label{prop12-finite-over-12}
Let $X$ be a Fano threefold with  Gorenstein canonical singularities. Assume that there are two   fibrations $f_1 \colon X\to \p^1$ and $f_2 \colon X\to \p^2$ such that $f=f_1\times f_2 \colon X\to \p^1\times \p^2$ is finite.  Assume that $K_X$ is numerically equivalent to the pullback of some $\mathbb{Q}$-divisor in $\p^1\times \p^2$ by $f$. Then $d=\mathrm{deg}\, f$ is at most $2$. Moreover, if $d=2$, then  $f$ is a double cover which is branched along a   divisor of degree $(2,k)$, where $k\in\{ 2,4\}$.
\end{prop}

\begin{proof}
Let  $A_1, A_2$ be two divisors in $X$ such that $\sO_X(A_1)\cong f_1^*\sO_{\p^1}(1)$ and $\sO_X(A_2)\cong f_2^*\sO_{\p^2}(1)$. By assumption,  there are two  rational numbers $a_1, a_2$ such that $-K_X \equiv a_1A_1+a_2A_2$.

We will first prove that $d\leqslant 2$.  We note that $$A_1A_2^2=d \ \mathrm{and}\  A_1^2A_2=A_1^3=A_2^3=0.$$  Since $f_2$ is a Fano fibration, general fibers of $f_2$ are   smooth rational curves. This implies that $-K_X\cdot A_2^2=2$. Hence, $a_1=\frac{2}{d}$.  Moreover, we have $$(-K_X)^3=3a_1a_2^2A_1A_2^2=3a_1a_2^2d=6a_2^2.$$ Since $-K_X$ is a  Cartier divisor, we obtain that    $6a_2^2\in \mathbb{Z}$.  Hence $a_2\in \mathbb{Z}$.

We have $(-K_X-a_2A_2)\equiv  \frac{2}{d}A_1$. Hence the Cartier divisor $-K_X-a_2A_2$ is numerically trivial on every fiber of $f_1$. Since $f_1$ is a  Fano fibration, we obtain that   $-K_X-a_2A_2$ is linearly equivalent to some integral multiple of $A_1$ by \cite[Lem. 3-2-5]{KMM87}. Thus $\frac{2}{d}\in \mathbb{Z}$ and $d\leqslant 2$.

If $d=1$, then $X\cong \p^1\times \p^2$. If $d=2$, then assume that $f$ is branched along a divisor $D$ of degree $(p,q)$. Then by Lemma \ref{lem-double-cover-Fano}, $X$ is Fano if and only if $p\leqslant 3$ and $q\leqslant 5$.  Since $f$ is a double cover, by Lemma \ref{lem-double-cover}, there is a line bundle $\sL$ on $\p^1\times \p^2$ such that $\sL^{\otimes 2} \cong \sO(D)$. Hence $p,q$ are even. Moreover, since $f_1,f_2$ have connected fibers, we have  $p,q>0$. Hence $f$ is branched along a divisor of degree $(2,2)$ or $(2,4)$. This completes the proof of the proposition.
\end{proof}

\begin{lemma}
\label{lem12-2finite-over-12}
Let $X$ be a Fano threefold with  Gorenstein canonical singularities. Assume that there is a   fibration  $f_1 \colon X\to \p^1$ and  a projective surjective morphism $f_2 \colon X\to \p^2$  such that general fibers of $f_2$ have two connected components. Assume that $f=f_1\times f_2 \colon X\to \p^1\times \p^2$ is finite. Let  $A_1, A_2$ be two divisors in $X$ such that $\sO_X(A_1)\cong f_1^*\sO_{\p^1}(1)$ and $\sO_X(A_2)\cong f_2^*\sO_{\p^2}(1)$. Assume that there are two  rational numbers $a_1, a_2$ such that $-K_X \equiv a_1A_1+a_2A_2$. Then $f$ is of degree $2$ or $4$ and $a_1, a_2\in \mathbb{Z}$.
\end{lemma}

\begin{proof}
Since general fibers of $f_2$ have two connected components, the degree of $f$ is an even integer. Set $ \mathrm{deg}\, f = 2d$.  Then we have $$A_1A_2^2=2d \ \mathrm{and}\  A_1^2A_2=A_1^3=A_2^3=0.$$  We note that each component of a general fiber of $f_2$  is a   smooth rational curve. This implies that $$-K_X\cdot A_2^2=2\times 2=4.$$ Hence, $a_1=\frac{2}{d}$.  Moreover, we have $$(-K_X)^3=3a_1a_2^2A_1A_2^2=6a_1a_2^2d=12a_2^2.$$ Since $(-K_X)^3$ is an even integer, we obtain that    $6a_2^2\in \mathbb{Z}$.  This implies that $a_2\in \mathbb{Z}$.

We have $(-K_X-a_2A_2)\equiv  \frac{2}{d}A_1$. As in the proof of Proposition \ref{prop12-finite-over-12}, we obtain that $\frac{2}{d}\in \mathbb{Z}$. Hence $d\leqslant 2$ and $a_1\in \mathbb{Z}$.
\end{proof}

\subsection{Finite covers of hypersurfaces of $\p^2 \times \p^2$}

In this subsection, we will prove some properties of finite morphism from a Fano threefold $X$ to an ample hypersurface $W$ in $\p^2\times \p^2$.

\begin{prop}
\label{prop12-finite-over-22}
Let $X$ be a Fano threefold with    Gorenstein canonical singularities. Assume that there are two  equidimensional fibrations $f_1 \colon X\to B_1$ and $f_2 \colon X\to B_2$ such that $B_1\cong B_2\cong\p^2$ and that $f_1\times f_2 \colon X\to B_1\times B_2$ is finite onto its image $W$. Assume that $K_X$ is numerically equivalent to the pullback of some $\mathbb{Q}$-divisor in $B_1\times B_2$ by $f_1\times f_2$.  Let $f \colon X\to W$ be the natural morphism.  Then there are exactly  two possibilities:

\begin{enumerate}

\item[(2)] $f$ is an isomorphism. $X$ is a hypersurface of degree $(k_1,k_2)$ with $k_1,k_2\in \{1,2\}$;

\item[(4)] $W$ is a smooth hypersurface of degree $(1,1)$, $f$ is a double cover branched along a   divisor  which is the intersection of $W$ and a divisor of degree $(2,2)$ in $B_1\times B_2$.
\end{enumerate}
\end{prop}

\begin{proof}
Let $A_1,A_2$ be two divisors in $X$ such that  $\sO_X(A_1)\cong f_1^*\sO_{\B_1}(1)$ and $\sO_X(A_2)\cong f_2^*\sO_{B_2}(1)$. By assumption,   there are two   rational numbers $a_1, a_2$ such that $-K_X\equiv a_1A_1+a_2A_2$.  Set $d=\mathrm{deg}\, f$.  Assume that $W$ is of degree $(p,q)$ in $B_1\times B_2$.  We have $$A_1^2A_2=dq \ \mathrm{and} \  A_1A_2^2=dp.$$ Moreover, $A_1^3=A_2^3=0$.  Since general fibers of  $f_1$ and $f_2$ are smooth rational curves,   we have  $-K_X\cdot A_1^2=-K_X\cdot A_2^2=2.$ Hence  $$a_1=\frac{2}{dp}\ \mathrm{and}\ a_2=\frac{2}{dq}.
$$

We have $$-K_X^3=3a_1a_2^2A_1A_2^2+3a_1^2a_2A_1^2A_2=\frac{24}{d^2}(\frac{1}{q^2}+\frac{1}{p^2}).$$  Since $-K_X^3$ a positive even integer, the number $$r=\frac{12}{d^2}(\frac{1}{q^2}+\frac{1}{p^2})$$ is a positive integer.  Without loss of generality, we may assume that $q\leqslant p$. 

If $d=1$, then $r$ is an integer if and only if $q,p \in \{1,2\}$.  The morphism $f \colon X\to W$ is the normalization map.  However, $W$ is normal by Lemma \ref{lem-normal-divisor}.  Thus $f$ is an isomorphism.

If   $d=2$, then we can only have $q=p=1$.  Since $f_1 \colon X\to B_1$ is equidimensional, so is the induced projection $p_1 \colon W\to B_1$. Hence every fiber of $p_1$ is a line in $\p^2$, which is smooth. Since $W$ is Cohen-Macaulay and $\p^2$ is smooth, the morphism $p_1$ is flat. Hence $p_1$ is a smooth morphism and $W$ is smooth. Since $W$ is ample in $\p^2\times \p^2$, the natural morphism $\mathrm{Pic}(\p^2\times \p^2) \to \mathrm{Pic}(W)$ is an isomorphism by the Lefschetz theorem (see \cite[Example 3.1.25]{Laz04}). Similarly to the last paragraph of the proof of Proposition \ref{prop12-finite-over-12}, we obtain that   $f$ is branched along a  divisor which is the intersection of $W$ and a divisor of degree $(2,2)$ in $B_1\times B_2$.

If $d\geqslant 3$, then we have  $$1\leqslant r\leqslant \frac{4}{3}(\frac{1}{q^2}+\frac{1}{p^2}).$$ Hence either $q=1$ or  $q,p\leqslant 2$. If $q=p=1$, then $r=\frac{24}{d^2}$. In this case, $r$ is an integer if and only $d=1$ or $2$. This is impossible since $d>2$. If $q=1$ and $p\geqslant 2$, then the previous inequality implies that $r\leqslant \frac{5}{3}$. Thus $r=1$ and   $p^2$ divides $12$. In this case, we can only have $p=2$ and $1=r=\frac{12}{d^2}\cdot \frac{5}{4}=\frac{15}{d^2}$. This is impossible.  If $q=p= 2$, then $r=\frac{6}{d^2}$   is an integer if and only if $d=1$. This is also a contradiction.   

Hence   $d\leqslant 2$. This completes the proof of the proposition.
\end{proof}

\begin{lemma}
\label{lem-22-finite-over-22}
Let $X$ be a Fano threefold with     Gorenstein canonical singularities. Let $B_1\cong B_2 \cong \p^2$. Assume that there are two equidimensional projective morphisms   $f_1 \colon X\to B_1$ and $f_2 \colon X\to B_2$ such that general fibers of $f_i$ have two connected components for $i=1,2$ and  $f_1\times f_2 \colon X\to B_1\times B_2$ is finite onto its image $W$. Let $g_i \colon X\to Z_i \to B_i$ be the Stein  factorization of $f_i$ for $i=1,2$.   Let $A_1,A_2$ be two divisors in $X$ such that  $\sO_X(A_1)\cong f_1^*\sO_{B_1}(1)$ and $\sO_X(A_2)\cong f_2^*\sO_{B_2}(1)$. Assume that

\begin{enumerate}
\item[-] $W$ is a normal hypersurface with isolated $\mathbb{Q}$-factorial klt singularities of degree $(p,q)$, where $p,q\in\{1,2\}$;
\item[-] $Z_1$ and $Z_2$ are smooth;
\item[-] there are two   rational numbers $a_1, a_2$ such that $-K_X\equiv a_1A_1+a_2A_2$. 
\end{enumerate}
Let $f \colon X\to W$ be the natural morphism.  Then  the degree of $f$ is $4$ and $a_1=a_2=p=q=1$.
\end{lemma}

\begin{proof}
Let $h$ be the natural morphism $Z_1\times Z_2 \to B_1\times B_2$ and let $Z=h^*W$. We will first show that that $Z$ is normal.  Since $Z_1$ is smooth, the double cover $Z_1\to B_1$ is a cyclic cover with respect to some isomorphism $(\sL_1)^{\otimes 2}\cong \sO_{B_1}(D_1)$, where $D_1$ is a smooth divisor (see Lemma \ref{lem-double-cover}). Let $p_1 \colon W\to B_1$ be the natural projection. Let $V\subseteq Z_1\times B_2$ be the pullback of $W$ by $Z_1\times B_2 \to B_1\times B_2$. Then $V\cong W\times_{B_1} Z_1$. 

Since $W$ is a Fano threefold with $\mathbb{Q}$-factorial klt  singularities,  $p_1$ is an equidimensional Fano fibration by Lemma \ref{lem-normal-hypersurface-equidim}.  Moreover, since $W$ has isolated singularities, by Lemma \ref{lem-smooth-Fano-surface-curve}, we obtain that $p_1$ has   reduced fibers over some open subset of $B_1$ whose complement has codimension at least $2$. Thus $V$ is normal and Cohen-Macaulay by Lemma \ref{lem-fiber-product-normal}.

\centerline{
\xymatrix{
Z\ar[d]\ar[rr] & & Z_2\ar[d]\\
 V \ar[r] \ar[d] & W \ar[r]_{p_2} \ar[d]_{p_1}& B_2 \\
 Z_1 \ar[r] & B_1  &  
}
}

Let $p_2 \colon W\to B_2$ and $r_2 \colon V\to B_2$ be the natural projections. Then there is an open subset $U$ of $B_2$ such that $\mathrm{codim}\, B_2\backslash U\geqslant 2$ and $p_2$ has reduced fibers over $U$ by Lemma \ref{lem-smooth-Fano-surface-curve} and  Lemma \ref{lem-normal-hypersurface-equidim}. We note that $V\to W$ is branched along $p_1^*D_1\subseteq W$. 
We claim that every irreducible component of $p_1^*D_1\subseteq W$ is horizontal over $B_2$. Indeed, if we assume that there is a prime divisor $E$ in $W$ which is both vertical over $B_1$ and $B_2$, then its intersections with the fibers of $p_1$  or   $p_2$ are all zero. Since $W$ is ample in $B_1\times B_2$, the natural morphism $\mathrm{Pic}(B_1\times B_2) \to \mathrm{Pic}(W)$ is an isomorphism by Lefschetz theorem (see \cite[Example 3.1.25]{Laz04}). Thus $E$ must be numerically zero for $W$ is $\mathbb{Q}$-factorial. This is a contradiction.

Since every irreducible component of $p_1^*D_1\subseteq W$ is horizontal over $B_2$, there is an open subset $U'$ of $U$ such that $\mathrm{codim}\, B_2\backslash U'\geqslant 2$ and $r_2$ has reduced fibers over $U'$. Since $Z$ is the pullback of $V$ by $Z_1\times Z_2 \to Z_1\times B_2$, we have $Z\cong V\times_{B_2} Z_2$.  As in the previous paragraph, Lemma \ref{lem-fiber-product-normal} implies that $Z$ is normal. 

We also obtain that $Z$ is irreducible.  Hence $Z$ is the image of $X$ in $Z_1\times Z_2$. Let $g \colon X\to Z$ be the natural morphism.  Let $d=\mathrm{deg}\, f$.  Since  $Z\to W$ is of degree $4$, we obtain  that $d\in 4\mathbb{Z}$. We have $$A_1^2A_2=qd \ \mathrm{and}\  A_1A_2^2=pd.$$ We note that every component of general fibers of $f_i$ is a smooth rational curve for $i=1,2$. Hence we have $$-K_XA_1^2=-K_XA_2^2=4.$$ This shows that $$a_1=\frac{4}{pd}, \ \mathrm{and}\  a_2=\frac{4}{qd}.$$ We have $$-K_X^3=3a_1a_2^2A_1A_2^2+3a_1^2a_2A_1^2A_2=\frac{192}{d^2}(\frac{1}{p^2}+\frac{1}{q^2}).$$ Since $-K_X^3$ is a positive even integer, we obtain that $$r=\frac{96}{d^2}(\frac{1}{p^2}+\frac{1}{q^2})$$ is an integer. Since $d\in 4\mathbb{Z}$, this implies that $d=4$ or $d=8$.

We note that $K_Z$ is Cartier for $Z$ is a Cartier divisor in $Z_1\times \Z_2$. Since $g \colon X\to Z$ is finite surjective, we obtain that $-K_Z$ is big by Lemma \ref{lem--K-big}.  We note that  $Z_1$ and $Z_2$ are smooth. By the adjunction formula and the ramification formula,  $-K_Z$ is   big if and only if $Z_i\to B_i$ is branched along a smooth conic  for $i=1,2$  and $p=q=1$.   Hence $Z_1\cong Z_2\cong \p^1\times \p^1$,  $Z$ is of degree $(1,1,1,1)$ in $Z_1\times Z_2\cong (\p^1)^{4}$ and $K_X\equiv \frac{4}{d} g^*K_Z$.    Hence by Lemma \ref{lem-p1p1p1p1-hypersurface}, we obtain that  $X\cong Z$ and $d=4$. Thus $a_1=a_2=1$.
\end{proof}

\begin{lemma}
\label{lem-12-finite-over-22}
Let $X$ be a Fano threefold with   Gorenstein canonical singularities. Let $B_1\cong B_2\cong \p^2$.  Assume that there are two equidimensional projective morphisms   $f_1 \colon X\to B_1$ and $f_2 \colon X\to B_2$ such that general fibers of $f_i$ have $i$ connected components for $i=1,2$ and  that $f_1\times f_2 \colon X\to B_1\times B_2$ is finite onto its image $W$.  Let $f \colon X\to W$ be the natural morphism.  Assume that 
\begin{enumerate}
\item[-] $W$ is a normal hypersurface with isolated $\mathbb{Q}$-factorial klt singularities of degree $(p,q)$, where $p, q \in \{1,2\}$;
\item[-] If $X \to Z_2\to B_2$ is the Stein factorization of $f_2$, then $Z_2$ is  smooth;
\item[-]  $K_X$ is  numerically equivalent to the pullback of some $\mathbb{Q}$-divisor in $B_1\times B_2$ by $f_1\times f_2$.
\end{enumerate}

 Then  the degree of $f$ is $2$   and $q=1$.
\end{lemma}

\begin{proof}
Let $Z=h^*W$, where $h$ is the natural morphism $B_1\times Z_2\to B_1\times B_2$. Then $Z\cong W\times_{B_2}Z_2$. Since $W$ has isolated klt singularities, $W\to B_2$ has reduced fibers over some open subset of $B_2$ whose complement has codimension at least $2$ by Lemma \ref{lem-smooth-Fano-surface-curve}. 

\centerline{
\xymatrix{
Z\ar[r] \ar[d] & W  \ar[d]_{} \\
Z_2 \ar[r] & B_2 
}
}
\noindent By Lemma \ref{lem-double-cover}, $Z_2\to B_2$ is a cyclic cover. Hence by Lemma \ref{lem-fiber-product-normal} and Lemma \ref{lem-normal-hypersurface-equidim}, $Z$ is normal. This also shows that $Z$ is irreducible.  Thus $Z$ equal to the image of $X$ in $B_1\times Z_2$.  Moreover, $Z$ is Gorenstein for it is a Cartier divisor in $B_1\times Z_2$.  Since $-K_X$ is ample, $-K_Z$ is big by Lemma \ref{lem--K-big}. Thus $q=1$ and  $Z_2\cong \p^1\times \p^1$ is a double cover of $B_2$, branched along some smooth conic.  We also note that if $d=\mathrm{deg}\, f$, then $d\in 2\mathbb{Z}$.

  Let $A_1,A_2$ be two divisors in $X$ such that  $\sO_X(A_1)\cong f_1^*\sO_{B_1}(1)$ and $\sO_X(A_2)\cong f_2^*\sO_{B_2}(1)$. By assumption, there are two   rational numbers $a_1, a_2$ such that $-K_X\equiv a_1A_1+a_2A_2$. We have $$A_1^2A_2=qd=d \ \mathrm{and}\  A_1A_2^2=pd.$$ We note that every component of a general fiber  of $f_i$ is a smooth rational curve for $i=1,2$. Hence we have $$-K_XA_1^2=2 \  \mathrm{and}\ -K_XA_2^2=4.$$ This shows that $$a_1=\frac{4}{pd} \ \mathrm{and}\  a_2=\frac{2}{d}.$$ We have $$-K_X^3=3a_1a_2^2A_1A_2^2+3a_1^2a_2A_1^2A_2=\frac{48}{d^2}(1+\frac{2}{p^2}).$$ Since $-K_X^3$ is a positive even integer, we obtain that $$r=\frac{24}{d^2}(1+\frac{2}{p^2})$$ is an integer. Since $d\in 2\mathbb{Z}$, this implies that $d=2$ or $6$.

Assume that $d=6$. Then $a_1=\frac{2}{3p}$ and $a_2=\frac{1}{3}$. The hypersurface $Z$ is of degree $(p,1,1)$ in $B_1\times Z_2\cong \p^2\times \p^1\times \p^1$ and $X\to Z$ is of degree $3$.  Let $\phi ,\psi$ be the natural projections from $X$ to $\p^1$ induced by $X\to Z_2$.  Let $F$ and $G$ be general  fibers of $\phi$ and $\psi$ respectively. Then $\sO_X(F+G)\cong \sO_X(A_2)$.  We have $$A_1^2F=3, \ A_1FG=3p, \ G^2F=0$$ and   $$-K_X\equiv \frac{2}{3p}A_1+\frac{1}{3}F+\frac{1}{3}G.$$ By the adjunction formula, we have
\begin{eqnarray*}
K_{F}^2&=&(K_X+F)^2F=(\frac{2}{3p}A_1+\frac{1}{3}G)^2F\\
&=&\frac{4}{9p^2}A_1^2F+\frac{4}{9p}A_1GF=\frac{4}{3p^2}+\frac{4}{3}=\frac{4}{3p^2}(p^2+1).
\end{eqnarray*}
Since $X$ has canonical singularities and $F$ is  a general fiber, $F$ is  a surface with canonical singularities. In particular, $K_F^2\in \mathbb{Z}$. This implies that $3$ divides $(p^2+1)$  which is impossible since $p\in \mathbb{Z}$.

Thus we have $d=2$.
\end{proof}

\subsection{Proof of Theorem \ref{thm-classification-3f}}

Theorem \ref{thm-classification-3f} follows from Proposition \ref{prop12-finite-over-12} and Proposition \ref{prop12-finite-over-22}.

\begin{proof}[{Proof of Theorem \ref{thm-classification-3f}}]
Let $R_1,R_2$ be the  two extremal rays of the cone $\NE(X)$. Let $f_i \colon X\to B_i$ be the Mori fibration corresponding  $R_i$ for $i=1,2$. Then by Lemma \ref{lem-base-fano}, $B_i$ is a Fano variety with Picard number $1$ for $i=1,2$.  By symmetry, we may assume that $\mathrm{dim}\, B_1\leqslant \mathrm{dim}\, B_2$. Since $R_1\cap R_2=\{0\}$, the morphism $f_1\times f_2 \colon X\to B_1\times B_2$ is finite onto its image. 

Assume first that  $B_1=\p^1$. Then $B_2$ is of dimension $2$ since $f_1\times f_2$ is finite. By Lemma \ref{lem-base-*-p1}, we have $B_2\cong \p^2$.  Since $X$ has Picard number $2$, by Proposition \ref{prop12-finite-over-12}, $X$ is one of the threefolds in Proposition \ref{prop12-finite-over-12}. Since $X$ is locally factorial, if $f$ is of degree two, then $f$ is branched along a prime divisor by \cite[Lem. 3.7]{Dru16}. Hence $X$ is one of the threefolds of $1-3$ of Theorem \ref{thm-classification-3f}.

Assume that $\mathrm{dim}\, B_1=2$. Then $\mathrm{dim}\, B_2=2$. By Lemma \ref{lem-base-*}, we have $B_1\cong B_2\cong \p^2$.  Then $X$ is one of the threefold in  Proposition  \ref{prop12-finite-over-22}. Since $X$ is locally factorial, as in the previous paragraph, we can conclude that $X$ is one of the threefolds of $4-7$ of Theorem \ref{thm-classification-3f}.
\end{proof}

The following observation  will be useful for the proof of Theorem \ref{thm-classification}.

\begin{lemma}
\label{lem-factorial-image-p2p2}
Let $X$ be a threefold in $4-7$ of Theorem \ref{thm-classification-3f}. Let $W$ be the image of the natural morphism $X\to \p^2 \times \p^2$. Then $W$ is a normal variety with isolated locally factorial canonical singularities. 
\end{lemma}

\begin{proof}
If the natural morphism $X\to W$ is an isomorphism, then there is nothing to prove. If $X\to W$ is of degree $2$, then $W$ is smooth of degree $(1,1)$ in $\p^2\times \p^2$.  
\end{proof}

\section{Case of dimension $4$}
\label{Fano fourfolds satisfying $(*)$}

We will study  Fano fourfolds $X$ with locally factorial canonical singularities such that $X$ is smooth in codimension $2$, $\rho(X) =3$ and  $\mathrm{Nef}(X) = \mathrm{Psef}(X)$. Consider the cone $\NE(X)$. For every face $V_2$ of dimension $2$ of $\NE(X)$, there are exactly two faces $V_1$ and $V_3$, of dimension $2$, whose intersections with $V_2$ are extremal rays of $\NE(X)$.   The aim of this section is to prove the following proposition.

\begin{prop}
\label{prop-4f-to-p1} 
If $X$ is a  Fano fourfold with locally factorial canonical singularities such that $X$ is smooth in codimension $2$, $\rho(X) = 3$ and  $\mathrm{Nef}(X) = \mathrm{Psef}(X)$, then there is a  fibration from $X$ to $\p^1$.
\end{prop}

We will assume the existence of Fano fourfolds which satisfy the following condition $(*)$ and we will obtain a contradiction.

\begin{enumerate}
\item[$(*)$] $X$ is a Fano fourfold with locally factorial canonical singularities such that $X$ is smooth in codimension $2$, $\rho(X) = 3$ and  $\mathrm{Nef}(X) = \mathrm{Psef}(X)$. Moreover, there is no fibration from $X$ to $\p^1$.
\end{enumerate}

We will first show  that if $X$ is a variety satisfying $(*)$, then it is a finite cover of the intersection  of two hypersurfaces of degree $(0,2,2)$ and $(2,2,0)$ in $\p^2\times \p^2 \times \p^2$.

\begin{lemma}
\label{lem-f4-p2}
Let $X$ be a Fano fourfold satisfying the conditions in Proposition \ref{prop-4f-to-p1}. Assume that there is no fibration from $X$ to $\p^1$. Let $V_1,V_2,V_3$ be three distinct  faces of dimensions $2$ in $\NE(X)$ such that $V_1\cap V_2$ and $V_2\cap V_3$ are extremal rays. Let $f_i \colon X\to B_i$ be the fibration corresponding to $V_i$ ($i=1,2,3$). Then $B_i= \p^2$ for all $i$ and the morphism $f_1\times f_2\times f_3 \colon X\to \p^2\times \p^2\times \p^2$ is finite onto its image $Z$. Moreover, $Z$ is the  intersection of hypersurfaces $D$ and $E$ of degree $(r,s,0)$ and $(0,a,b)$ respectively such that  $r,s,a,b\in \{1,2\}$.
\end{lemma}

\begin{proof}
For all $i$, since the base $B_i$ has Picard number $1$ and there is no  fibration from $X$ to $\p^1$, by Lemma \ref{lem-base-*},  we obtain that $B_i\cong \p^2$.

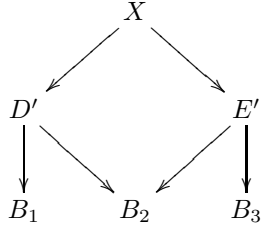
\centerline{
\xymatrix{
& X \ar[rd] \ar[ld] &\\
 D' \ar[d] \ar[rd]& &E' \ar[d]\ar[ld] \\
B_1& B_2 & B_3
}
}

Let $g \colon X\to Y$ be the Mori fibration  corresponding to the extremal ray  $R_{12}=V_1\cap V_2$. Then $g$ is equidimensional of relative dimension $1$ by Lemma \ref{lem-equidimensional}.  By Lemma \ref{lem-base-*}, $Y$ is a Fano threefold with isolated locally factorial canonical singularities such that  $\rho(Y) = 2$ and  $\mathrm{Nef}(Y) = \mathrm{Psef}(Y)$.  Hence $Y$ is one of the threefolds in $4-7$ of Theorem \ref{thm-classification-3f}. Let $D'$ be the image of $f_1\times f_2 \colon X\to B_1 \times B_2$. Since $f_1\times f_2$ factorizes through $g \colon X\to Y$,  $D'$ is a   hypersurface of degree $(r,s)$ in $B_1 \times B_2$  with $r,s\in \{1,2\}$ by Theorem \ref{thm-classification-3f}. By the Lefschetz theorem, the natural map $\mathrm{Pic}(B_1\times B_2) \to \mathrm{Pic}(D')$ is an isomorphism  (see \cite[Example 3.1.25]{Laz04}). Moreover, $D'$ has isolated locally factorial canonical singularities  by Lemma \ref{lem-factorial-image-p2p2}.
 
Let $h \colon X\to D'$ be the natural morphism. Since $R_{12} \cap V_3 =\{0\}$, the morphism $h\times f_3 \colon X\to D'\times B_3$ is finite onto its image $Z$.   Let  $D=D'\times B_3$. Then $D$ is a hypersurface of degree $(r,s,0)$ in $B_1\times B_2\times B_3$. Since $B_3$ is smooth, the product $D$ is locally factorial (this follows from the fact that the ring of polynomials with coefficients in a unique factorization domain is still a unique factorization domain). Hence $Z$ is a Cartier divisor in $D$.

We note that the natural morphism $ \mathrm{Pic}(D')\times \mathrm{Pic}(B_3) \to \mathrm{Pic}(D)$ is an isomorphism since $B_3=\p^2$ is a simply connected manifold (see \cite[Ex. III.12.6]{Har77}). Hence the natural morphism $$\mathrm{Pic}(B_1)\times \mathrm{Pic}(B_2)\times \mathrm{Pic}(B_3) \to  \mathrm{Pic}(D)$$ is an isomorphism. Since $Z$ is a Cartier divisor in $D$, we obtain  that there is a hypersurface $E$ of degree $(c,a,b)$ in $B_1\times B_2\times B_3$ such that $Z$ is the intersection of $E$ and $D$.

Let $R_{23}$ be the extremal ray $V_2\cap V_3$.  As in the first paragraph, we obtain that  the image of the morphism $f_2\times f_3 \colon X\to B_2\times B_3$  is a proper subvariety. Thus  the image of the natural morphism $Z\to  B_2\times B_3$ is also a proper subvariety. This implies that $c=0$.

Let $E'\subseteq B_2\times B_3$ be the image of $f_2\times f_3$. Since $f_2\times f_3$ factors through $X\to Z$, we obtain that $E'$ is of degree $(a,b)$ in $B_2\times B_3$.  As in the second paragraph of the proof, we have $a,b\in\{1,2\}$ and $E\cong B_1\times E'$. This completes the proof of the lemma.
\end{proof}

\begin{lemma}
\label{lem-4f-sim-p1}
Let $X$ be a Fano fourfold satisfying the conditions in Proposition \ref{prop-4f-to-p1}. Assume that there is no fibration from $X$ to $\p^1$. Then the cone $\NE(X)$ is not simplicial.
\end{lemma}

\begin{proof}
Assume the opposite.  Since the cone $\NE(X)$ is simplicial, it has exactly three faces $V_1$, $V_2$, $V_3$ of dimension $2$ and exactly three extremal rays $ V_1\cap V_2$, $ V_2\cap V_3$ and  $V_3\cap V_1$. As in the proof of Lemma \ref{lem-f4-p2}, for every face $V_i$, there is a  fibration $f_i \colon X\to B_i$ corresponding $V_i$, where $B_i\cong \p^2$ ($i=1,2,3$). Let $Z$ be the image of $f_1\times f_2\times f_3 \colon X\to B_1\times B_2\times B_3$. Then $Z$ is the  intersection of hypersurfaces of degree $(r,s,0)$ and $(0,a,b)$ and that $r,s,a,b \in \{1,2\}$. Since   $V_3\cap V_1\neq \{0\}$, we obtain that the image of  $f_1\times f_3 \colon X\to B_1\times B_3$ is a proper subvariety. As in the proof of Lemma \ref{lem-f4-p2}, this implies that $Z$ is contained in a hypersurface of degree $(a',0,b')$ for some $a',b'\in \mathbb{Z}$. This is a contradiction for $a,s\neq 0$.
\end{proof}

\begin{lemma}
\label{lem-p2p2p2-1}
Let $W=B_1\times \B_2\times B_3$, where $B_i=\p^2$ for $i=1,2,3$. Let $D'$ and $E'$ be  normal $\mathbb{Q}$-factorial klt hypersurfaces in $B_1\times B_2$ and $B_2\times B_3$ of degree $(a,1)$ and $(b,c)$ respectively with $a,b,c\in \{1,2\}$.  Let $D= D'\times B_3$ and $E= B_1\times E'$. In  other words, $D$ and $E$ are hypersurfaces of degree $(a,1,0)$ and $(0,b,c)$ in $W$. Let $Z$ be the intersection of $D$ and $E$.  If $f_1,  f_3$ are the natural projections from $Z$ to  $B_1, B_3$ respectively, then $f=f_1\times f_3 \colon Z\to B_1\times B_3$ is not finite.
\end{lemma}

\begin{proof}

Let $p_1$, $p_2$ (resp. $q_2$, $q_3$) be the natural projections from $D'$ (resp. $E'$) to $B_1$ and $B_2$ (resp. $B_2$ and $B_3$). Since  $D'$ and $E'$ are $\mathbb{Q}$-factorial klt Fano varieties, the morphisms $p_1,p_2,q_2,q_3$ are equidimensional by Lemma \ref{lem-normal-hypersurface-equidim}. Thus $p_1 \colon D'\to B_1$ is a well-defined family of cycles in $B_2=\p^2$ (see \cite[Def. I.3.11]{Kol96}). Since $D'$ is of degree $(a,1)$, the cycles in this family are lines in $B_2$. Thus there exist a natural morphism $ch$ from $B_1$ to the Grassmannian of lines in $B_2$ (see \cite[Thm. I.3.21]{ Kol96}), which is   isomorphic to $\p^2$. Since $a\neq 0$, the morphism $ch$ is not constant. We note that $B_1$ has Picard number $1$. This shows that $ch$ is  finite. Hence it is surjective and for any line $L$ in $B_2$, there is a point $x$ in $B_1$ such that the fiber of $p_1$ over $x$  is mapped to  $L$ by the projection $p_2$.

\centerline{
\xymatrix{
& Z= (D'\times B_3) \cap (B_1\times E')  \ar[rd] \ar[ld] &\\
 D' \ar[d]^{p_1} \ar[rd]^{p_2}& &E' \ar[d]^{q_3}\ar[ld]^{q_2} \\
B_1& B_2 & B_3
}
}

In order to prove the lemma, it is enough to find two points $x\in B_1$ and $z\in B_3$ such that  $f^{-1}(\{(x,z)\})$  contains a curve. First we assume that $b=1$. Let $L$ be a line in $B_2$. From the previous paragraph, there is a point  $x\in B_1$   such that $p_1^{-1}(\{x\})$ is mapped to $L$ by $p_2$. Similarly, there is a point $z\in B_3$ such that  $q_3^{-1}(\{z\})$ is mapped to $L$ by   $q_2$.  Hence the curve $\{(x,y,z)\in W \ | \ y \in L \}$ is contained in $f^{-1}(\{(x,z)\})$.  Thus, $f$ is not finite.

Now we assume that $b=2$. Since $q_3$ is equidimensional, $q_3 \colon E'\to B_3$ is a well-defined family of cycles in $B_2\cong \p^2$. Hence there is a morphism from $B_3$ to the Chow variety  of conics in $\p^2$, which is isomorphic to $\p^5$. Since $B_3$ has Picard number $1$ and $c\neq 0$, this morphism is  finite onto its image.   We note that there is a hypersurface $H$ in this Chow variety  which parametrizes the singular conics. Moreover, since $\p^5$ has Picard number $1$, the divisor $H$ is ample. Thus $H$ meets the image of $B_3$.  In particular, there is a point $z\in B_3$ such that the   fiber of $q_3$ over $z$ is the union of two lines $L$ and $L'$. There is a point $x\in B_1$ such that $p_1^{-1}(\{x\})$ is mapped to $L$ by $p_2$. This shows that the curve $\{(x,y,z)\in W \ | \ y \in L \}$ is contained in $f^{-1}(\{(x,z)\})$. Hence $f$ is not finite.
\end{proof}

\begin{lemma}
\label{lem-f4-p2-2}
With the notation in Lemma \ref{lem-f4-p2}, the numbers $a,b,r,s$ are all equal to $2$. 
\end{lemma}

\begin{proof}
By Lemma \ref{lem-4f-sim-p1}, the cone $\NE(X)$ is not simplicial. Hence $V_1\cap V_3=\{0\}$. This implies that $f_1\times f_3 \colon X \to B_1 \times B_3$ is finite surjective.  Recall that   $ Z =(f_1\times f_2 \times f_3)(X) \subseteq B_1 \times B_2 \times B_3.$ Since $f_1\times f_3$ factors through $X\to Z$, the natural morphism $Z\to B_1\times B_2$ is finite surjective. We can write $D=D'\times B_3$ and $E=B_1\times E'$. Then $D'$ and $E'$ are Fano varieties with isolated locally factorial canonical singularities (see  the second paragraph of the proof of Lemma \ref{lem-f4-p2}). Thus by Lemma \ref{lem-p2p2p2-1}, we have  $s=a=2$. 

We note that there is a face $V_4$ of dimension $2$ of $\NE(X)$ such that $V_2\neq V_4$ and $V_3\cap V_4$ is an extremal ray.   Let   $f_4 \colon X\to B_4$ be the fibration corresponding to $V_4$. Then, as in Lemma \ref{lem-f4-p2}, $B_4=\p^2$ and $$f_2\times f_3\times f_4 \colon  X\to B_2\times B_3 \times B_4$$ is finite onto its image $T$. Moreover, $T$ is the intersection  of hypersurfaces of degree of $(a,b,0)$ and $(0,p,q)$  such that   $p,q\in \{1,2\}$. As in the previous paragraph, we have $b=p=2$. Similarly, we can obtain that $r=2$.
\end{proof}

\begin{lemma}
\label{lem-image-normal-4f}
With the notation in Lemma \ref{lem-f4-p2}, the image $Z$ of $f_1\times f_2\times f_3 \colon X\to B_1\times B_2\times B_3$ is normal.
\end{lemma}

\begin{proof} 
By Lemma \ref{lem-f4-p2-2}, we have $Z=D\cap E$, where $D\cong D'\times B_3$ and $E\cong B_1\times E'$ are hypersurfaces of degree $(2,2,0)$ and $(0,2,2)$ respectively in $B_1\times B_2\times B_3$.

By Lemma \ref{lem-factorial-image-p2p2},  $E'$ is a Fano threefold  with isolated singularities. Hence, by Lemma \ref{lem-smooth-Fano-surface-curve}, there is an open subset $U$ of $B_2$ whose complement is of codimension at least $2$ such that the natural projection $q_2 \colon E'\to B_2$ has reduced fibers over $U$.

\centerline{
\xymatrix{
& Z= (D'\times B_3) \cap (B_1\times E')  \ar[rd] \ar[ld]^{\phi} &\\
 D' \ar[d]^{} \ar[rd]^{p_2}& &E' \ar[d]^{}\ar[ld]^{q_2} \\
B_1& B_2 & B_3
}
}
Consider the natural projection $\phi \colon Z\to D'$. Let $\alpha$ be a point in $D'$. Then the fiber of $\phi$ over $\alpha$ is a conic in $B_3$. Moreover, if $\beta$ is the image of $\alpha$ by the natural projection $p_2 \colon D'\to B_2$, then the fiber of $q_2$ over $\beta$ is isomorphic to  the fiber of $\phi$ over $\alpha$. Hence general fibers of $\phi$ are smooth conics in $B_3$ and the fibers of $\phi$ over $p_2^{-1}(U)\subseteq D'$ are reduced.

We note that $D'$ is a Fano threefold with locally factorial canonical singularities by Lemma \ref{lem-factorial-image-p2p2}. Thus  $p_2 \colon D'\to B_2$ is equidimensional by Lemma \ref{lem-normal-hypersurface-equidim}. Since $\mathrm{codim}\, B_2\backslash U\geqslant 2$, the complement of $p_2^{-1}(U)$ in $D'$ has codimension at least $2$. Let $$V=\{z\in Z\ |\ \phi^{-1}(\{\phi(z)\}) \ \mathrm{is \ singular \ at}\ z \}.$$ Then $\mathrm{codim}\, V\geqslant 2$. Since $Z$ is Cohen-Macaulay and $\phi$ is equidimensional, $\phi$ is flat over the smooth locus of $D'$. Since $D'$ is smooth in codimension $1$, this implies that $Z$ is smooth in codimension $1$. Thus $Z$ is normal for it is Cohen-Macaulay.
\end{proof}

Now we can prove Proposition \ref{prop-4f-to-p1}.

\begin{proof}[{Proof of Proposition \ref{prop-4f-to-p1}}]
Assume the opposite. With the same notation as in Lemma \ref{lem-f4-p2}, there are fibrations $f_i \colon X\to B_i$  such that $f_1\times f_2\times f_3 \colon X\to B_1\times B_2\times B_3$ is finite onto its image $Z$. By Lemma \ref{lem-f4-p2-2}, $Z$ is the intersection of $D'\times B_3$ and $B_1\times E'$ such that $D'$ is normal of degree $(2,2)$ in $B_1\times B_2$ and $E'$ is normal of degree $(2,2)$ in $B_2\times B_3$ (see Lemma \ref{lem-factorial-image-p2p2}). 

By Lemma \ref{lem-image-normal-4f}, $Z$ is normal.  Let $p_i \colon Z\to B_i$ be the natural fibration and let $H_i$ be a divisor in $Z$ such that $\sO_Z(H_i) \cong p_i^*\sO_{B_i}(1)$ for $i=1,2,3$. Then $$H_1^2H_3^2= H_1^2H_2H_3=H_1H_2H_3^2=4,$$ and,  by adjunction formula, $$-K_Z\equiv H_1-H_2+H_3.$$ Since $X$ is a Fano variety, $-K_Z$ is big by Lemma \ref{lem--K-big}. Moreover, since $\NE(X)$ is not simplicial by Lemma \ref{lem-4f-sim-p1}, we obtain that $V_1\cap V_3=\{0\}$ and that $f_1\times f_3 \colon X\to B_1\times B_3$ is finite. This shows that $p_1\times p_3 \colon Z\to B_1\times B_3$ is also finite. Hence $H_1+H_3$ is an ample divisor and the intersection number $-K_Z(H_1+H_3)^3$ is positive. However, we have 
\begin{eqnarray*}
-K_Z(H_1+H_3)^3& =&(H_1+H_3-H_2)(3H_1^2H_3+3H_1H_3^2)\\
&=&6H_1^2H_3^2-3H_1^2H_2H_3-3H_1H_2H_3^2=0. 
\end{eqnarray*}
We obtain a contradiction.
\end{proof}

\section{Finite morphisms between Fano varieties.}

Consider a variety $X$ of dimension $n$ which satisfies the conditions in Theorem \ref{thm-classification}. If we assume Theorem \ref{thm-classification} in lower dimensions,  then  we can show that there is a finite morphism from $X$ to $(\p^1)^{n-2}\times \p^2$ or to $(\p^1)^{n-3}\times W$, where $W$ is an ample hypersurface in $\p^2 \times \p^2$ (see the last section for a complete argument). The objective of this section is then to study    finite morphisms  from a Fano variety  to $(\p^1)^{n-2}\times \p^2$ or to $(\p^1)^{n-3}\times W$.

\subsection{Finite morphisms  over $(\p^1)^{n-2}\times \p^2$}

We will first look at  finite morphisms from an $n$-dimensional Fano variety  to  $(\p^1)^{n-2}\times \p^2$.

\begin{lemma}
\label{lem-finite-over4-112}
Let $Y$ be a variety  constructed by the method of  Construction   \ref{exmp-112}.2. Set $n=\mathrm{dim}
\, Y$. Assume that $Y$ is a variety with locally factorial canonical singularities such that it is smooth in codimension $2$, $\rho(Y) = \mathrm{dim}\, Y-1$ and  $\mathrm{Nef}(Y) = \mathrm{Psef}(Y)$. Let $X$ be a Fano variety with Gorenstein canonical singularities such that  there is a finite surjective morphism $f \colon X\to Y$. Assume that   the projection  from $X$ to $(\p^1)^{n-2}$  induced by $f$ has connected fibers.  Then $f$ is an isomorphism.
\end{lemma}

\begin{proof}
From the construction of $Y$, we have two double covers $$Y\overset{\pi_1}{\longrightarrow} T \overset{\pi_2}{\longrightarrow} (\p^1)^{n-2}\times \p^2,$$ where $\pi_2 \colon T\to  (\p^1)^{n-2}\times \p^2$ is branched along some prime divisor of degree $(2,...,2,0,...,0,2)$. In particular, $T\cong (\p^1)^r\times T'$, where $0<r<n-2$ and  $T'$ is a double cover over $(\p^1)^{n-2-r}\times \p^2$, branched along a divisor of degree $(2,...,2,2)$.  We note that $Y\to T' \to (\p^1)^{n-2-r}\times \p^2$ is the Stein factorization. By Lemma \ref{lem-base-*-p1}, we obtain that $T'$ is smooth in codimension $2$. Thus $T$ is smooth in codimension $2$.

We write $$(\p^1)^{n-2}\times \p^2=(\p^1)^{n-3}\times (\p^1 \times \p^2) = (\p^1)^{n-3}\times Z.$$ Let $p$ be the natural morphism from $Y$ to $(\p^1)^{n-3}$. Let $F$ be the fiber of $p$ over a general point $s\in (\p^1)^{n-3}$. By restricting $\pi_1$ and $\pi_2$ over the point $s$, we obtain two double covers  $$F\overset{\pi_1|_F}{\longrightarrow} V \overset{\pi_2|_V}{\longrightarrow} \{s\}\times Z,$$  where $V$ is the fiber of the natural projection $T\to (\p^1)^{n-3}$ over $s$. The double cover $\pi_2|_V \colon V\to \p^1\times \p^2$ (here we identify $\{s\}\times Z$ with $\p^1\times \p^2$) is branched along a prime divisor of degree $(0,2)$. Thus $V\cong \p^1\times V'$, where $V'$ is a double cover over $\p^2$. Since $T$ is smooth in codimension $2$, $V$ is also smooth in codimension $2$. This implies that $V\cong \p^1\times V'$ is in fact smooth. Thus the branched locus of $\pi_2|_V$ is  smooth  by  Lemma \ref{lem-double-cover}. 

Hence $V\cong \p^1\times \p^1 \times \p^1$. Assume that $\pi_1|_F\colon F\to V$ is branched along a divisor $D$. Then $D$ is the pullback of some   divisor of degree $(2,2)$ in $Z$ by construction. Thus $D$ is a   divisor of degree $(2,2,2)$ in $V\cong \p^1\times \p^1 \times \p^1$. By construction of $Y$, the double cover $\pi_1 \colon Y\to T$ is branched along a prime divisor. Moreover, since $Y$ is smooth in codimension $2$, this branched locus itself is smooth in codimension $1$ by Lemma \ref{lem-double-cover}.  Thus $D$  is also smooth in codimension $1$. Since $D$ is connected, it is then  irreducible.

Let $G$ be the fiber of the natural morphism $X\to (\p^1)^{n-3}$ over $s$. Then $f|_G \colon G\to F$ is also a finite surjective morphism. Since $s$ is a general point, $G$  is a Fano variety with Gorenstein canonical singularities. By \cite[Thm. 1.3]{Dru16}, $f|_G$ is of degree $1$. Hence $f$ is of degree $1$ and is an isomorphism.
\end{proof}

\begin{lemma}
\label{lem-finite-over-112}
Let $X$ be a Fano variety with locally factorial canonical singularities such that $X$ is smooth in codimension $2$, $\rho(X) = \mathrm{dim}\, X-1$ and  $\mathrm{Nef}(X) = \mathrm{Psef}(X)$. Set $n=\mathrm{dim}\, X$ and assume that $n\geqslant 3$. Assume that Theorem \ref{thm-classification} is true in dimension smaller than $n$. Assume that there is a fibration $X\to (\p^1)^{n-2}$ and a fibration $X\to \p^2$ such that the product morphism $f \colon X\to (\p^1)^{n-2} \times \p^2$ is finite surjective.  Then $X$ is isomorphic to a product as described in Theorem \ref{thm-classification}.
\end{lemma}

\begin{proof}
Let $Z=(\p^1)^{n-3}$ be the product of the last $n-3$ factors in the product $(\p^1)^{n-2}$. We note that if $X\to Y$ is the Stein factorization of the projection $$X\to Z\times \p^2 = (\p^1)^{n-3}\times \p^2,$$ then $Y$ is a Fano variety with locally factorial canonical singularities such that $Y$ is smooth in codimension $2$, $\rho(Y) = \mathrm{dim}\, Y-1$ and  $\mathrm{Nef}(Y) = \mathrm{Psef}(Y)$ by Lemma \ref{lem-base-*-p1}. Moreover, since $X\to (\p^1)^{n-2}$ has connected fibers, the induced morphism $Y\to Z$ has connected fibers. Therefore,  by Theorem \ref{thm-classification} in dimension $n-1$, we have $Y\cong (\p^1)^{s-1} \times Y_1$, where $s\geqslant 1$ and $Y_1$ is  a variety constructed  by the methods of  Construction   \ref{exmp-112}.  Thus  $Y\to Z \times \p^2$ is of degree $1$, $2$ or $4$. Moreover, by Lemma \ref{lem-base-*-p1}, $Y_1$ is a Fano variety with locally factorial canonical singularities such that $Y_1$ is smooth in codimension $2$, $\rho(Y_1) = \mathrm{dim}\, Y_1-1$ and  $\mathrm{Nef}(Y_1) = \mathrm{Psef}(Y_1)$. Let $f_1 \colon X\to \p^1$ be the projection to the  first factor of $(\p^1)^{n-2}$.

Let $A_1$ is be a fiber of $f_1$. Let $A_2,...,A_{n-2}$ be   fibers of $X$ over the other different $\p^1$. Let $H$ be the pullback of a line in $\p^2$ by the natural projection $X\to \p^2$. Since $\rho(X)=n-1$, there are positive rational numbers $a_1,...,a_{n-2},b$ such that $-K_X\equiv a_1A_1+\cdots +a_{n-2}A_{n-2}+bH$.

\centerline{
\xymatrix{
X\ar[rd]_{f_1}  \ar[r] &  V=\p^1\times Y \ar[r]\ar[d] &Y \cong (\p^1)^{s-1}\times Y_1 \ar[r]  & Z\times \p^2 \\
&\p^1 & 
}
}

We will discuss   three cases.

\textit{Case 1.}  Assume that the morphism $Y \to Z\times \p^2$ is of degree $4$.  Then $Y_1$ is a variety which can be constructed by the method of    Construction   \ref{exmp-112}.2 and $X$ is finite over $(\p^1)^{s}\times Y_1$.   The natural morphism  $X\to (\p^1)^{s}$  is a fibration since $X\to (\p^1)^{n-2}$ is.  Let $F$ be a general fiber of $X\to (\p^1)^{s}$. Then $F$ is a Fano variety with Gorenstein canonical singularities. Since $X\to (\p^1)^{n-2}$ has connected fibers,  the natural morphism $F\to (\p^1)^{n-2-s}$ induced by $X\to Y_1\to (\p^1)^{n-2-s}$ is a fibration. By Lemma \ref{lem-finite-over4-112}, the natural finite morphism $F \to Y_1$ is an isomorphism. Hence $X\cong (\p^1)^{s}\times Y_1$ in this case.
  
\textit{Case 2.} Assume that the morphism $Y\to Z\times \p^2 $ is of degree $2$. Let $G$ be a general fiber of the natural projection $X\to Z$. Then $G$ is a Fano threefold with Gorenstein canonical singularities. Moreover, the restriction $f_1|_G \colon G\to \p^1$ is a fibration,  general fibers of the  natural morphism $G\to \p^2$ have two connected components, and the product of these two morphisms $G\to \p^1\times \p^2$ is finite.  We note that $K_G\equiv a_1A_1|_G+bH|_G$. By Lemma \ref{lem12-2finite-over-12}, $G\to \p^1\times \p^2$ is of  degree $2$ or $4$ and $b\in \mathbb{Z}$. Since the degree of $f$ is the same as the one of $G\to \p^1\times \p^2$, it is either $2$ or $4$. 

The Cartier divisor $-K_X-bH$ is numerically trivial on the fibers of the Fano fibration $X\to (\p^1)^{n-2}$.  Thus it is linearly equivalent to the pullback of some Cartier divisor in $(\p^1)^{n-2}$ by \cite[Lem. 3-2-5]{KMM87}. This implies that $a_1,...,a_{n-2}\in \mathbb{Z}$. 

We note that   $Y_1$ is a double cover of $(\p^1)^{r}\times \p^2$, branched along some prime divisor of degree $(2,...,2,k)$ such that  $k\in \{2,4\}$  and $r+s=n-2$. If $\mathrm{deg}\, f=2$, then $X\cong  (\p^1)^{s} \times Y_1$. 

If $\mathrm{deg}\, f=4$, then $X$ is a double cover of $V= \p^1\times Y=(\p^1)^{s} \times Y_1$.  There is a natural double cover $h \colon V\to (\p^1)^{s} \times (\p^1)^{r}\times \p^2$, induced by $Y_1\to (\p^1)^r\times \p^2$, branched along a prime divisor of degree $(0,...,0 ,2,...,2,k)$.  By the ramification formula, $K_V$ is linearly equivalent to the pullback by $h$ of a divisor of degree $(-2,...,-2,-1,...,-1,-3+k/2)$ in $ (\p^1)^{s} \times (\p^1)^{r}\times \p^2$.  Since  $\rho(V)=n-1$, we may assume that the double cover $X\to V$ is branched along a divisor $D\subseteq V$ which is numerically equivalent to the pullback by $h$ of a  divisor of degree $(d_1,...,d_s,d'_1,...,d'_r,e)$ in  $(\p^1)^{s} \times (\p^1)^{r}\times \p^2$. By the ramification formula again, $K_X$ is linearly equivalent to the pullback of a divisor of degree $(-2+d_1/2,...,-2+d_s/2,-1+d'_1/2,...,-1+d'_r/2,-3+k/2+e/2)$ in $(\p^1)^{s} \times (\p^1)^{r}\times \p^2$.  Since $a_1,...,a_{n-2},b\in \mathbb{Z}$ and $-K_X$ is ample,  we obtain that $d_1,...,d_s, e\in \{0,2\}$ and that $d_1'=...=d_r'=0$. Since $X\to (\p^1)^{n-2}$ has connected fibers, we obtain that $e>0$. Hence $e=2$. Since $b>0$, we obtain that $k=2$. Moreover, since  $X\to \p^2$ has connected fibers, $d_1,...,d_s$  are not all zero. We note that $V$ is simply connected and locally factorial. Thus $D$ is in fact linearly equivalent to the pullback of a  divisor of degree $(d_1,...,d_s,0,...,0,e)$ in  $(\p^1)^{s}\times \p^2 \times (\p^1)^{r}$. We note that $D$ is a prime divisor by \cite[Lem. 3.7]{Dru16}. Thus we obtain that $X\cong X_1\times X_2$, where $X_1=(\p^1)^{t}$ with $s>t\geqslant 0$ and $X_2$ is a variety which can be constructed by the method of  Construction   \ref{exmp-112}.2.

\textit{Case 3.} Assume   $Y\to Z\times \p^2 $ is an isomorphism. Let $G$ be a general fiber of the natural projection $X\to Z$. Then $G$ is a Fano threefold with Gorenstein canonical singularities. Moreover,  $f_1|_G \colon G\to \p^1$ and the natural morphism $G\to \p^2$ are fibrations, and the product $h \colon  G\to \p^1\times \p^2$ is finite. By Proposition \ref{prop12-finite-over-12}, $h$ is of degree at most $2$. Hence $f$ is of degree at most $2$.

If the degree of $f$ is $1$, then $f$ is an isomorphism. If $\mathrm{deg}\, f=2$, then by Lemma \ref{lem-double-cover} and Lemma \ref{lem-double-cover-Fano}, $f \colon X\to (\p^1)^{n-2} \times \p^2$  is a double cover, branched along a   divisor of degree $(d_1,...,d_{n-2},k)$ such that $d_1,...,d_{n-2}\in \{0,2 \}$ and that $k\in \{0,2,4\}$. Since $X$ is locally factorial, the branched locus of $f$ is a prime divisor by \cite[Lem. 3.7]{Dru16}. Since $X\to (\p^1)^{n-2}$ and $X\to \p^2$ have connected fibers, we obtain that $k\neq 0$ and $d_1,...,d_{n-2}$ are not all zero. Hence $X \cong X_1\times X_2$, where  $X_1 \cong (\p^1)^{t}$ for some $t\geqslant 0$ and $X_2$ is a double cover of $(\p^1)^{n-2-t}\times \p^2$, branched along a prime divisor  of degree $(2,...,2,2)$ or $(2,...,2,4)$. 
\end{proof}

\subsection{Finite morphisms over hypersurfaces of $(\p^1)^{n-3}\times \p^2\times \p^2$} In this subsection, we will study finite morphisms from a $n$-dimensional Fano variety  onto $(\p^1)^{n-3}\times W$, where $W$ is an ample hypersurface in $\p^2 \times \p^2$.

\begin{lemma}
\label{lem-hypersurface-4-122-non-fano}
Let $Y$ be a variety constructed by the method of    Construction   \ref{exmp-112}.2   of dimension $n$. Then there is a finite surjective morphism, of degree $4$, $Y\to (\p^1)^{n-2}\times B_1$ with $B_1\cong \p^2$.  Let $X$ be a  hypersurface in $Y\times B_2$ with $B_2\cong \p^2$. Assume that the image of $X$ in $B_1\times B_2$ by the natural projection is a normal hypersurface $W$ of degree $(p,q)$ with isolated $\mathbb{Q}$-factorial klt singularities, where  $p,q\in \{1,2\}$.    Then $X$ is normal and  $-K_X$ is not big.  
\end{lemma}

\begin{proof}
There is a natural finite surjective morphism  $X\to D$, where  $$D=(\p^1)^{n-2}\times W \subseteq (\p^1)^{n-2}\times B_1\times B_2.$$   By the construction of $Y$, the natural morphism $Y\times B_2\to (\p^1)^{n-2}\times B_1 \times B_2$ is the composition of the following two double covers $$Y \times B_2  \overset{g}{\longrightarrow} V \times B_2 \overset{h}{\longrightarrow} (\p^1)^{n-2}\times B_1\times B_2,$$ where $h$ is a double cover, branched along a prime divisor of degree $(0,...,0,2,...,2,2,0)$.

We note that the natural morphism $D\to (\p^1)^{n-2}\times B_1$ is an equidimensional  Fano fibration since $W\to B_1$ is equidimensional by Lemma \ref{lem-normal-hypersurface-equidim}. Since $W$ has isolated singularities, $D$ is smooth in codimension $2$. There is an open subset $U$ in $(\p^1)^{n-2}\times B_1$ whose complement has codimension at least $2$ such that  $D\to (\p^1)^{n-2}\times B_1$ has reduced fibers over $U$ by Lemma \ref{lem-smooth-Fano-surface-curve}. We note that $V\to (\p^1)^{n-2}\times \B_1$ is a cyclic cover  branched along some prime divisor and that $$h^*D\cong D\times_{(\p^1)^{n-2}\times B_1} V.$$ Hence, $h^*D$ is normal and Cohen-Macaulay by Lemma \ref{lem-fiber-product-normal}.   Moreover,  $h^*D\to V$ also has reduced fibers over some open subset of $V$ whose complement has codimension at least $2$.  Similarly, by Lemma \ref{lem-fiber-product-normal} again,  $$g^*(h^*D)= (h^*D)\times_ {V} Y$$ is also normal. Hence, $X=g^*(h^*D)$ and it is normal.

\centerline{
\xymatrix{
X  =g^*(h^*D)  \ar[d] \ar[r]  & h^*D \ar[d] \ar[r]  & D \ar[d]\\
 Y \ar[r] &  V  \ar[r]&  (\p^1)^{n-2}\times B_1
}
}

Let $A_1,...,A_{n-2}$ be  fibers of the natural projections from $X$ to the different factors $\p^1$. Let $H_i$   be a divisor  in $X$ which is the pullback  of a line in $B_i$ for $i=1,2$. Then by the adjunction formula, we have $$-K_X\equiv A_1+\cdots+A_{n-2}+(1-p)H_1+(3-q)H_2.$$ Since $p\geqslant 1$, $-K_X$ is not big.
\end{proof}

\begin{lemma}
\label{lem-finite-over-122}
Let $X$ be a Fano variety with locally factorial canonical singularities such that $X$ is smooth in codimension $2$, $\rho(X) = \mathrm{dim}\, X-1$ and  $\mathrm{Nef}(X) = \mathrm{Psef}(X)$. Set $n=\mathrm{dim}\, X$ and assume that $n\geqslant 4$. Assume that Theorem \ref{thm-classification} is true in dimension smaller than $n$. Assume that there are  fibrations $X\to (\p^1)^{n-3}$,  $X\to B_1$ and  $X\to B_2$  such that 
\begin{enumerate}
\item $B_1\cong B_2\cong \p^2$,
\item $X\to B_1\times B_2$ is not surjective, 
\item $X\to (\p^1)^{n-3} \times B_1 \times B_2$ is finite onto its image.
\end{enumerate}
Then $X$ is isomorphic to a product as described in Theorem \ref{thm-classification}.
\end{lemma}

\begin{proof}
Let $W$ be the image of $X\to B_1\times B_2$ and let $X\to W'\to W$ be the Stein factorization. Since the natural morphism $X\to (\p^1)^{n-3} \times W'$ is finite surjective, by Lemma \ref{lem-base-*-p1}, $W'$ is a Fano threefold of Theorem \ref{thm-classification-3f}. Thus by Lemma \ref{lem-factorial-image-p2p2}, $W$ is a normal hypersurface of degree $(p,q)$, with isolated locally factorial canonical singularities,  such that $p,q\in \{1,2\}$.  Let $f \colon X\to (\p^1)^{n-3}\times W$ be the natural morphism.

Let $X\to Y$ be the Stein factorization of $X\to (\p^1)^{n-3}\times B_1$. Then $X\to Y$ is a Mori fibration of relative dimension $1$ (see Lemma \ref{lem-relative-1-Mori}). By   Lemma \ref{lem-base-fano}, $Y$ is a Fano variety with locally factorial canonical singularities such that   $\rho(Y) = \mathrm{dim}\, Y-1$ and  $\mathrm{Nef}(Y) = \mathrm{Psef}(Y)$. The fibration $X\to Y$ is equidimensional since $(\p^1)^{n-3}\times W\to (\p^1)^{n-3}\times  B_1$ is by Lemma \ref{lem-normal-hypersurface-equidim}. Hence by  Lemma \ref{lem-Mori-fibration-relative-1}, $Y$ is smooth in codimension $2$.  Thus, by Theorem \ref{thm-classification} in dimension $n-1$, $Y$ is isomorphic to a product as described in Theorem \ref{thm-classification}. Moreover, since $X\to (\p^1)^{n-3}$   has connected fibers, so has the induced morphism $Y\to (\p^1)^{n-3}$. Thus   $Y\cong (\p^1)^r \times Y_1$ for some $r\geqslant 0$ such that  $ Y_1$ is either $\p^2$  or a variety which can be constructed by the method of Construction \ref{exmp-112}. In particular,  $Y \to (\p^1)^{n-3}\times B_1$ is of degree $1$, $2$ or $4$.

If it is of degree $4$, then $Y=(\p^1)^r\times Y_1$  where $r\geqslant 0$ and $Y_1$ is a variety which can be constructed by the method of  Construction   \ref{exmp-112}.2.  Let $S$ be a general fiber of $X\to (\p^1)^r$. Then $S$ is a Fano variety. If $T$ is the image of $S$ in $Y_1\times \B_2$, then  by Lemma \ref{lem-hypersurface-4-122-non-fano}, $T$ is normal and $-K_T$ is not big. Thus $S$ cannot be a Fano variety by Lemma \ref{lem--K-big}. This is a contradiction. Hence the degree of $Y \to (\p^1)^{n-3}\times B_1$ is $1$ or $2$. This shows that  fibers of $X\to (\p^1)^{n-3}\times B_1$ have at most two connected components.  Similarly,  fibers of $X\to (\p^1)^{n-3}\times B_2$ have at most two connected components.

Let $A_1,...,A_{n-3}$ be   fibers of the natural fibrations from $X$ to  different $\p^1$. Let $H_i$ be the pullback  of a line in $B_i$ by the  fibration  from $X$ to $B_i$ for $i=1,2$. Since the  Picard number of $X$ is $n-1$, there are positive rational numbers $a_1,...,a_{n-3},b_1,b_2$ such that $-K_X\equiv a_1A_1+\cdots a_{n-3}A_{n-3}+b_1H_1+b_2H_2$. Let $F$ be the fiber of $X\to (\p^1)^{n-3}$ over a general point $\alpha\in (\p^1)^{n-3}$, then $-K_F \equiv (b_1H_1+b_2H_2)|_F$. Moreover, $F$ is a Fano threefold with isolated Gorenstein canonical singularities and the natural morphism $F\to W$ is finite.

We will discuss   four cases.

\textit{Case 1.} Assume that   the  projection   $X\to (\p^1)^{n-3} \times B_i$ is a  fibration for $i=1,2$.  Then the  natural projections from $F$ to  $B_1$ and $B_2$ are fibrations.  Since $F$ is finite over $W$ and the natural morphisms $W\to B_1$ and $W\to B_2$ are equidimensional (see Lemma \ref{lem-normal-hypersurface-equidim}), the fibrations $F\to B_1$ and $F \to B_2$ are equidimensional. Hence by Proposition \ref{prop12-finite-over-22}, the finite morphism $F\to W$ is of degree $1$ or $2$. 

If this degree is $1$, then $f \colon X\to (\p^1)^{n-3}\times W$ is an isomorphism. If the degree of $F\to W$ is $2$, then $f$ is a double cover. Moreover, $W$ is smooth of degree $(1,1)$ in $B_1\times B_2$ by Proposition \ref{prop12-finite-over-22}. We note that there is a natural isomorphism $\mathrm{Pic}((\p^1)^{n-3}\times B_1\times B_2) \cong \mathrm{Pic}((\p^1)^{n-3}\times W)$ by the Lefschetz theorem (see \cite[Example 3.1.25]{Laz04}). By Lemma \ref{lem-double-cover} and  Lemma \ref{lem-double-cover-Fano}, $f$ is branched along a prime divisor $D$ which is the intersection of $(\p^1)^{n-3}\times W$ and a divisor of degree $(d_1,...,d_{n-3},k,l)$ in $(\p^1)^{n-3}\times B_1\times B_2$ such that $d_1,...,d_{n-3},k,l \in \{0,2\}$. Since the two morphisms $F\to B_1$ and $f\to B_2$ are fibrations, $k,l\neq 0$. Hence $X\cong X_1\times X_2$, where $X_1=(\p^1)^{r}$ with some $r\geqslant 0$ and $X_2$ is a variety which can be constructed by the method of   Construction   \ref{exmp-122}.1.

\textit{Case 2.} Assume that  $X\to (\p^1)^{n-3}\times B_2$ is a fibration and general fibers of  $X\to (\p^1)^{n-3}\times B_1$ have two connected components. We recall that $X\to Y$ is the Stein factorization of $X\to (\p^1)^{n-3}\times B_1$. Then $Y=(\p^1)^r \times Y_1$, where $r\geqslant 0$ and $Y_1$ is a variety which can be constructed by the method of   Construction   \ref{exmp-112}.1. If $F\to Z_1\to B_1$ is the  Stein factorization, then $Z_1$ is isomorphic to the fiber of $Y\to (\p^1)^{n-3}$ over $\alpha$. Since $Y$ is smooth in codimension $2$, $Z_1$ is  smooth. Moreover, as in Case 1, the morphisms from $F$ to $B_1$ and $B_2$ are equidimensional.   Thus $F$ satisfies the conditions in Lemma \ref{lem-12-finite-over-22}. This implies that $p=1$ and $F\to W$ is of degree $2$. Thus $X\to (\p^1)^{n-3}\times W$ is of degree $2$.

Let $V$ be the pullback of $(\p^1)^{n-3}\times W$ by the natural morphism $Y\times B_2 \to (\p^1)^{n-3}\times B_1\times B_2$. We note that  the fibration  $(\p^1)^{n-3}\times W\to (\p^1)^{n-3}\times B_1$ has reduced fibers over some open subset of $(\p^1)^{n-3}\times B_1$ whose complement has codimension at least $2$, for $W$ is smooth in codimension $2$ (see Lemma \ref{lem-smooth-Fano-surface-curve}).  Thus  $V$, which is isomorphic to the fiber product $((\p^1)^{n-3}\times W) \times_{(\p^1)^{n-3}\times B_1} Y$, is  normal  by Lemma \ref{lem-fiber-product-normal} and irreducible.  Since $X\to (\p^1)^{n-3}\times W$ is of degree $2$, we obtain that $X\cong V$.  Since $X$ is a Fano variety, by ramification formula and adjunction formula, we obtain that the double cover $Y_1\to (\p^1)^{n-3-r}\times B_1$ is branched along some divisor of degree $(2,...,2,2)$. Hence   $X\cong X_1\times X_2$, where $X_1=(\p^1)^{r}$  and $X_2$ is a variety which can be constructed by  the method of  Construction   \ref{exmp-122}.2.

\textit{Case. 3.}  Assume that $X\to (\p^1)^{n-3}\times B_1$ is a fibration and general fibers of  $X\to (\p^1)^{n-3}\times B_2$ have two connected components. By symmetry, we can reduce to  the second case. 

\textit{Case 4.} Assume that general fibers of the two projections   $X\to (\p^1)^{n-3}\times B_1$ and $X\to (\p^1)^{n-3}\times B_2$ have two connected  components.  As in the second case, if  $F\to Z_i\to B_i$ is the  Stein factorization for $i=1,2$, then $Z_1$ and $Z_2$ are smooth. Hence $F$ satisfies the conditions in Lemma \ref{lem-22-finite-over-22}. The morphism $F\to W$ is of degree $4$ and $p=q=b_1=b_2=1$. Hence the morphism $X\to (\p^1)^{n-3}\times W$ is of degree $4$. Since the Cartier divisor $-K_X-b_1H_1-b_2H_2$ is numerically trivial on the fibers of the Fano fibration $X\to (\p^1)^{n-3}$, we obtain that $a_1,...,a_{n-3}\in \mathbb{Z}$ by \cite[Lem. 3-2-5]{KMM87}.

We recall that  $X\to Y$ is the Stein factorization of $X\to (\p^1)^{n-3}\times B_1$. Then we have $Y=(\p^1)^r \times Y_1$, where   $r\geqslant 0$ and $Y_1$ is a variety which can be constructed by the method of  Construction   \ref{exmp-112}.1. Let $V$ be the pullback of $(\p^1)^{n-3}\times W$ by the natural morphism  $Y\times B_2 \to  (\p^1)^{n-3}\times B_1\times B_2$.    Then, as in the second case,  $V$ is  normal and irreducible. Hence $X\to V$ is a double cover.

\centerline{
\xymatrix{
X     \ar[r]  & V \ar[d] \ar[r]  & (\p^1)^{n-3}\times W  \ar[d] \ar[r]& (\p^1)^{n-3}\times B_2\\
  &  Y  \ar[r]&  (\p^1)^{n-3}\times B_1 &
}
}

Similarly, there is an open subset $U$ of $(\p^1)^{n-3}\times B_2$ whose complement has codimension at least $2$ such that the   $(\p^1)^{n-3}\times W \to (\p^1)^{n-3} \times B_2$ has reduced fibers over $U$. The branch locus of $V\to (\p^1)^{n-3}\times W$ is the intersection of  $(\p^1)^{n-3}\times W$ and a divisor of degree $(0,...,0,2,...,2,2,0)$ in $(\p^1)^{n-3} \times B_1 \times B_2$. Thus this branch locus is  horizontal over $(\p^1)^{n-3} \times B_2$. Hence, the natural morphism $V\to (\p^1)^{n-3} \times B_2$  is a fibration and it has reduced fibers over some open subset $U'\subseteq U$ such that $\mathrm{codim}\,  U\backslash U'\geqslant 2$.  Thus  the pullback of any prime divisor in $(\p^1)^{r} \times B_2$ by the fibration $V\to (\p^1)^{r} \times B_2$ is reduced.

Assume that $X\to V$ is branched along some   divisor $D$. Then $D$  is a prime divisor by \cite[Lem. 3.7]{Dru16} for $X$ is locally factorial. Since $a_1,...,a_{n-3}\in \mathbb{Z}_{>0}$ and $b_1=b_2=1$, by the adjunction formula and the ramification formula, we obtain that $Y_1\to (\p^1)^{n-3-r}\times \B_1$ is branched along some divisor of degree $(2,..,2,2)$ and $D$ is numerically equivalent to the pullback of a divisor $E$ of degree $(d_1,...,d_r,2)$ in $(\p^1)^{r} \times B_2$ by the natural projection $V\to   (\p^1)^{r} \times B_2$ such that $d_1,...,d_r\in \{0,2\}$.  Since the pullback of any prime divisor in $(\p^1)^{r} \times B_2$ by the fibration $V\to (\p^1)^{r} \times B_2$ is reduced,  we have $D=f^*(f(D))$ and $D$ is in fact linearly equivalent to the pullback of $E$.

Since the projection  from $X $ to $B_2$ is a fibration, we obtain that $r\geqslant 1$ and that $d_1,...,d_r$ are not all equal to $0$. Hence we obtain that $X\cong X_1\times X_2$, where $X_1=(\p^1)^{t}$ for some $r>t\geqslant 0$ and $X_2$ is a variety which can be constructed by the method of   Construction   \ref{exmp-122}.3.
\end{proof}

\section{Proof of Theorem \ref{thm-classification}}

We will first prove the following lemma.

\begin{lemma}
\label{lem-existence-fibration-p1}
Let $X$ be a Fano  variety with locally factorial canonical singularities such that $X$ is smooth in codimension $2$, $\rho(X) = \mathrm{dim}\, X-1\geqslant 3$ and  $\mathrm{Nef}(X) = \mathrm{Psef}(X)$. Then there is a  fibration  from  $X$ to $\p^1$.
\end{lemma}

\begin{proof}
We will argue by contradiction. Set $n=\mathrm{dim}\, X$. Assume   that there is no fibration from $X$ to $\p^1$. Let $V$ be a face of $\NE(X)$ of dimension $n-4$. Then there is a morphism $X\to Y$ corresponding to $V$. By Lemma \ref{lem-base-*}, $Y$ is a Fano variety with locally factorial canonical singularities such that $Y$ is smooth in codimension $2$, $\rho(Y) = \mathrm{dim}\, Y-1$ and  $\mathrm{Nef}(Y) = \mathrm{Psef}(Y)$. We note that $\rho(Y)=(n-1)-(n-4)=3$. Hence $\mathrm{dim}\, Y=4$. By Proposition \ref{prop-4f-to-p1},  there is a fibration from $Y$ to $\p^1$ which  induces a fibration from $X$ to $\p^1$. This is a contradiction. 
\end{proof}

Now we can prove Theorem \ref{thm-classification}.

\begin{proof}[{Proof of Theorem \ref{thm-classification}}]

We will proceed by induction on the dimension of $X$. If $\mathrm{dim}\, X=2$, then $X\cong \p^2$. In this case, we let   $X_1$ be a point and we let $X_2$ be $\p^2$. If $\mathrm{dim}\, X=3$, then the theorem follows from Theorem \ref{thm-classification-3f}. Assume that the theorem is true in dimension smaller than $n$, where $n\geqslant 4$ is an integer.   Now we consider the case of $\mathrm{dim}\, X =n$.

By Lemma \ref{lem-existence-fibration-p1}, there is a fibration $f_1 \colon X\to \p^1$.  There is an extremal ray $R$ in $\NE(X)$ such that the class of any curve in $X$ contracted by $f_1$ is not contained in $R$. Let $f_2 \colon X\to Y$ be the fibration corresponding to $R$. Then $f=f_1\times f_2 \colon  X\to \p^1\times Y$ is a finite surjective morphism. By assumption, Theorem \ref{thm-classification} is true in dimension smaller than $n$. Hence $Y$ is isomorphic to a product as described in Theorem \ref{thm-classification} by Lemma \ref{lem-base-*-p1}. We have the following two possibilities.

\textit{Case 1.} The variety $Y$ is a finite cover of $(\p^1)^{n-3}\times \p^2$ and  $f$ induces a finite surjective morphism $g \colon X\to (\p^1)^{n-2}\times \p^2$ such that the projections from $X$ to $\p^1$ and the projection $g_2$ from $X$ to $\p^2$ induced by $g$ are fibrations. Let $X\overset{g_1}{\longrightarrow} Z\overset{\pi}{\longrightarrow} (\p^1)^{n-2}$  be the Stein factorization of the natural projection induced by $g$.      Then $Z$ is one of the varieties in Theorem \ref{thm-double-cover-p1} by Lemma \ref{lem-base-fano}, Lemma \ref{lem-relative-1-Mori} and Lemma \ref{lem-base-Picard-number}. 

If $\pi$ is of degree at least $2$, then we can write $Z\cong Z_1\times Z_2$ such that $Z_1$ is a double cover of $(\p^1)^{r}$ for some $r\geqslant 3$, branched along some prime divisor of degree $(2,...,2)$, and $Z_2$ is one of the varieties of Theorem \ref{thm-double-cover-p1}, of dimension $n-2-r$. Let $h_1 \colon X\to Z_1$ be the natural fibration. Let $h_2 \colon X\to X'$ be the Stein factorization of the natural projection $X\to Z_2\times \p^2$. Then, by Lemma \ref{lem-base-*-p1}, $X'$ is a Fano variety  with locally factorial canonical singularities such that $X'$ is smooth in codimension $2$, $\rho(X') = \mathrm{dim}\, X'-1$ and  $\mathrm{Nef}(X') = \mathrm{Psef}(X')$. Hence it is isomorphic to a product as described in Theorem \ref{thm-classification} by  induction hypotheses. Let $F$ be a general fiber of $h_2$. Then $h_1|_F \colon F \to Z_1$ is a finite surjective morphism. Since $F$ is a Fano variety with canonical Gorenstein singularities, and since $Z_1$ is a double cover  of $(\p^1)^r$ branched along some prime divisor of degree $(2,...,2)$, we obtain that  $h_1$ is an isomorphism by \cite[Thm. 1.3]{Dru16}. Hence $X\cong Z_1\times X'$, and it is isomorphic to a product as described in Theorem \ref{thm-classification}.

If $\pi$ is of degree $1$, then $Z= (\p^1)^{n-2}$. By Lemma \ref{lem-finite-over-112},   $X$ is isomorphic to a product as described in Theorem \ref{thm-classification}.

\textit{Case 2.} The variety $Y$ is a finite cover of $(\p^1)^{n-4}\times W$, where $W$ is a normal $\mathbb{Q}$-factorial divisor of degree $(p,q)$ in $B_1\times B_2\cong \p^2\times \p^2$ such that $p,q\in \{1,2\}$. Moreover,  $f$ induces a finite surjective morphism $g \colon X\to (\p^1)^{n-3}\times W$ such that the projections from $X$ to $\p^1$ and the projections  from $X$ to $B_1$ and $B_2$ induced by $g$ are fibrations. Let $X\overset{g_1}{\longrightarrow} Z\overset{\pi}{\longrightarrow} (\p^1)^{n-3}$  be the Stein factorization of the natural projection induced by $g$.   As in the first case, if $\pi$ is of degree at least $2$, then $X\cong Z_1\times X'$, where $Z_1$ is a double cover of $(\p^1)^{r}$ for some $r\geqslant 3$, branched along some prime divisor of degree $(2,...,2)$, and $$X\to X'\to (\p^1)^{n-3-r}\times W$$ is the Stein factorization. By Lemma \ref{lem-base-*-p1} and induction hypotheses, $X'$ is isomorphic to a product as described in Theorem \ref{thm-classification}. Thus $X$ is also isomorphic to a product as described in Theorem \ref{thm-classification}.

If $\pi$ is of degree $1$, then $Z=(\p^1)^{n-3}$. By Lemma \ref{lem-finite-over-122},   $X$ is isomorphic to a product as described in Theorem \ref{thm-classification}. This completes the proof of the theorem.
\end{proof}

\bibliographystyle{amsalpha}
\bibliography{references}

\end{document}